\documentclass[11pt,english]{amsart}

\usepackage{amsmath,leftidx,amsthm}
\usepackage[all]{xy}
\usepackage{xspace}
\usepackage[psamsfonts]{amssymb}
\usepackage[latin1]{inputenc}
\usepackage{graphicx,color,mathrsfs}
\usepackage{hyperref,appendix}
\usepackage{xcolor}

\newtheorem{theorem}{\bf Theorem}[section]
\newtheorem{proposition}[theorem]{\bf Proposition}
\newtheorem{definition}[theorem]{\bf Definition}

\newtheorem{example}[theorem]{\bf Example}

\newtheorem{lemma}[theorem]{\bf Lemma}
\newtheorem{corollary}[theorem]{\bf Corollary}

\newtheorem{remark}[theorem]{\bf Remark}

\def\C{{\mathbb C}}

\def\R{{\mathbb R}}
\def\Z{{\mathbb Z}}

\def\P{\mathbb{P}}

\def\supp{\textup{supp}}

\title{Exponential mixing of all orders and CLT for generic birational maps of $\P^k$}
\author{Henry De Thélin}
\address{LAGA, UMR 7539, Institut Galilée, Université Paris 13, 99 avenue J.B. Clément, 93430 Villetaneuse, France}
\email{dethelin@math.univ-paris13.fr}
\author{Gabriel Vigny}
\address{LAMFA UMR 7352, Universit\'e de Picardie Jules Verne, 33 rue Saint-Leu, 80039 AMIENS Cedex 1, FRANCE}
\email{gabriel.vigny@u-picardie.fr}
\thanks{The second author's research is partially supported by the ANR QuaSiDy, grant [ANR-21-CE40-0016] }

\begin{document}
\begin{abstract} For Hénon maps, Bianchi and Dinh \cite{BianchiDinh1} recently proved the exponential mixing of all orders for  the measure of maximal entropy and, as a consequence of the recent work of Bj{\"o}rklund and Gorodnik \cite{BG}, the CLT  for H\"older observables. 
	We extend their results to generic birational maps of $\P^k$. Because of the indeterminacy set, H\"older maps are not stable under iteration, so we need to work with a suitable space of test functions.  
\end{abstract}

\maketitle

\medskip

\noindent \textbf{Mathematics~Subject~Classification~(2020): 37A25, 37A50, 37F80}

\section{Introduction}

Let $(X,f)$ be a dynamical system and $\nu$ an $f$-invariant probability measure.  Take $\varphi : X \to \R$ an observable, if the system is chaotic enough, one can expect that the sequence $\varphi \circ f^n$ is asymptotically a sequence of random independent variables (they are however not independent, since they arise from a deterministic setting). For example, if $\nu$ is ergodic and  $\varphi \in L^1(\nu)$, Birkhoff 
theorem asserts that for $\nu$-a.e. $x$
\begin{equation*}
	\frac{1}{n} S_n (\varphi)  (x):= \frac{1}{n} (\varphi (x) + \varphi\circ f (x) + \dots + \varphi\circ f^{n-1} (x)) \to \langle \nu, \varphi \rangle,
\end{equation*} 
which can be thought of as the law of large numbers. The next steps are then to try to understand the speed of mixing i.e. the speed at which $\langle  \nu, \varphi\circ f^n \psi \rangle -\langle \nu,  \varphi \rangle  \langle \nu, \psi \rangle$ goes to zero for observables $\varphi$ and $\psi$ (or more general correlations) and to establish a central limit theorem for the sequence $(\varphi \circ f^n)$. \\

In the present article, we are interested in these questions for a class of birational maps $f:\P^k(\C)\dasharrow \P^k(\C)$ endowed with their Green measure. Such maps are holomorphic outside an indeterminacy set, denoted by  $I_f$, and admits a rational inverse, denoted by $f^{-1}$, outside an analytic set. Just like birational maps of $\P^2(\C)$ generalize Hénon maps, the maps we consider are natural generalizations of Hénon-Sibony maps, which are the polynomials automorphisms of $\C^k$ for which $I_{f^{-1}} \cap I_{f}=\varnothing$. Their dynamics is very well understood, see~\cite[Chapitre 2]{Sibony}. We make the assumptions that, as for Hénon-Sibony maps,
there is an integer $1\leq s\leq k-1$ such that 
\begin{align}
	\dim I_f=k-s-1 \quad \text{and} \quad \dim I_{f^{-1}}=s-1,\tag{$\dag$}\label{gooddim}
\end{align}
 and
\begin{align}
	\bigcup_{n\geq0}f^n(I_{f^{-1}})\cap I_f=\bigcup_{n\geq0}f^{-n}(I_{f})\cap I_{f^{-1}}=\varnothing.\tag{$\star$}\label{AS}
\end{align}
Let $d$ (resp. $\delta$) be the algebraic degree of $f$ (resp. $f^{-1}$), i.e. the degree of the homogeneous polynomials defining $f$. Under the hypothesis \eqref{gooddim} and \eqref{AS}, one can show that $d^s=\delta^{k-s}$ and that the algebraic degree of $f^n$ (resp. $f^{-n}$) is $d^n$ (resp. $\delta^n$), see~\cite{Dethelin_Vigny}. Inspired by the work of Bedford and Diller \cite{Bedford-Diller} in dimension $2$, we introduced for such maps a finite energy condition ~\cite[Definition 3.1.9]{Dethelin_Vigny}, which can be stated as
\begin{align}\label{energy_higher2}
	\left\{\begin{array}{l}
		\displaystyle \sum^{\infty}_{n=0} \frac{1}{d^{sn}} \int_{f^n(I_{f^{-1}})} u^+ \cdot f^*(\omega^{s-1}) >-\infty  \  \text{and} \\
		\displaystyle  \sum^{\infty}_{n=0} \frac{1}{\delta^{n(k-s)}} \int_{f^{-n}(I_f)} u^- \cdot (f^{-1})^*(\omega^{k-s-1}) >-\infty, 
	\end{array}	\right.
\end{align}
where $\omega$ is the Fubini-Study form on $\P^k(\C)$ and $u^+$ (resp. $u^-$) is a quasi-potential of $d^{-1} f^*(\omega)$, i.e. $d^{-1} f^*(\omega)=\omega + dd^c u^+$ (resp. a quasi-potential of $\delta^{-1} (f^{-1})^*(\omega)$, i.e. $\delta^{-1} (f^{-1})^*(\omega)=\omega + dd^c u^-$). Condition \eqref{energy_higher2} is satisfied by Hénon-Sibony maps and regular birational maps of $\P^k$ (\cite{DinhSibonyR}).
 Under condition \eqref{energy_higher2}, we constructed a mixing and hyperbolic measure $\mu$ of maximal entropy $s \log d$ and pluripolar sets are of $\mu$-measure zero, see~\cite{Dethelin_Vigny}. We call $\mu$ \emph{the Green measure} of $f$. Recall a set is pluripolar if, locally, it is contained in the set $\{v=-\infty\}$ for some plurisubharmonic (psh for short) function $v$. 
 As said above, such maps where previously introduced in dimension $2$ by Bedford and Diller \cite{Bedford-Diller} who already constructed the measure in that case. 

We will call \emph{generic birational map of $\P^k$} a map satisfying \eqref{gooddim}, \eqref{AS} and \eqref{energy_higher2}. These maps are generic in the following senses:
\begin{itemize}
	\item if $f$ is a birational maps satisfying \eqref{gooddim}, then, outside a pluripolar set of maps $A \in \mathrm{PSL}(k+1, \C)$, $A \circ f$ satisfies  \eqref{AS} and \eqref{energy_higher2} (\cite{Bedford-Diller} in dimension $2$ and \cite{Dethelin_Vigny} in higher dimension).
	\item if $f$ is a birational map of $\P^2$	satisfying \eqref{gooddim} and \eqref{AS} and $f$ is defined over a number field then $f$ satisfies \eqref{energy_higher2} (\cite{Jonsson-Reschke} in dimension 2 and \cite{gauthier2023complex} in higher dimension).
\end{itemize}

Take now $\varphi : \P^k \to \R$ an observable. As explain above, the purpose of this article is to show that the sequence of random variables $\varphi \circ f^n$ is asymptotically independent in a quantified way. We are mainly interested in $C^\alpha$ observables and in differences of bounded quasi-plurisubharmonic (qpsh for short) functions. Differences of qpsh functions, usually called DSH functions, were introduced by Dinh and Sibony (see \cite{Dinh_Sibony_distrib} and the survey \cite{dinhsibony2}). They are important in complex dynamics as their definition takes into account the complex structure. More precisely, take $a >0$ and consider the space of DSH functions $\varphi \in L^\infty$ with $dd^c \varphi = T_1-T_2$, where $T_i$ are positive closed currents such that $T_i=\|T_i\|( \omega+ dd^c u_i)$ with $u_i$ qpsh and $0\geq u_i \geq -a$ for $i=1,2$. We denote by $\mathrm{DSH}_a^\infty(\P^k)$ this space of test functions, which we endow with the norm
\[ \|\varphi\|_{\mathrm{DSH}_a^\infty(\P^k)}=\|\varphi\|_\infty + \inf\{ \|T_i\|, T_i \ \mathrm{as \ above}\}.\] 
We simply write $\mathrm{DSH}^\infty(\P^k)$ for a difference of bounded qpsh functions. Our first main results is the following. 
\begin{theorem}[Exponential mixing of all orders]\label{tm_exp} Let $f$ be a generic birational map. Take $a>0$. For all $ p\geq 1$ and $\alpha\leq 2$, there exists a constant $C_{p,\alpha} >0$ such that $ \forall \varphi^0, \dots, \varphi^p \in C^{\alpha}(\P^k)$ and $\forall 0=:n_0\leq n_1 \leq \dots \leq n_p$, we have 
	\begin{align*}
		&\left|  \langle \mu, \prod_{j=0}^p \varphi^j\circ f^{n_j} \rangle -   \prod_{j=0}^p\langle \mu, \varphi^j\rangle \right|\leq 	C_{p,\alpha} d^{-\frac{\alpha^{p+1}}{2^{p+1}} \frac{s}{2k}\min_{0\leq j \leq p-1}(n_{j+1}-n_j)} \prod_{i=0}^p \|\varphi^i\|_{C^\alpha} 
	\end{align*}		
	and there exists a constant $C_p$ such that for $ \forall \varphi^0, \dots, \varphi^p \in \mathrm{DSH}_a^\infty(\P^k)$ and $\forall 0=:n_0\leq n_1 \leq \dots \leq n_p$, we have 
	\begin{align*}
		&\left|  \langle \mu, \prod_{j=0}^p \varphi^j\circ f^{n_j} \rangle -   \prod_{j=0}^p\langle \mu, \varphi^j\rangle \right|\leq 	C_{p} d^{- \frac{s}{2k}\min_{0\leq j \leq p-1}(n_{j+1}-n_j)} \prod_{i=0}^p  \|\varphi^i\|_{\mathrm{DSH}_a^\infty(\P^k)}. 
	\end{align*}		
\end{theorem}	
For Sibony-Hénon maps, this result was first proved  for $C^\alpha$-observables by Dinh \cite{Dinh_decay} for $p=1$ and recently by Bianchi and Dinh \cite{BianchiDinh1} for arbitrary $p$, and for $p=1$ by Wu for bounded DSH observables \cite{Hao_decay} (note that Wu also proved a version for non bounded DSH observables though, in our case, DSH functions could be non-integrable with respect to $\mu$). For generic birational maps, the second author showed the exponential mixing of order $p=1$ for $C^\alpha$-observables \cite{Vigny_decay}. \\

Let us now give an application of Theorem~\ref{tm_exp} to the statistical properties of $f$. For $\varphi\in L^1 (\mu)$, recall that we write 
 \[ S_n (\varphi)(x):=
 \big( \varphi (x)+ \varphi \circ f(x) + \dots + \varphi \circ f^{n-1}(x) \big). \]
  We say that $\varphi$ \emph{satisfies the Central Limit Theorem}
 (CLT)
 with variance $\sigma^2 \geq 0$ with respect to $\mu$
 if $n^{-1/2} (S_n (\varphi)  - n\langle\mu,\varphi\rangle )\to \mathcal N(0,\sigma^2)$ in law,
 where $\mathcal N (0,\sigma^2)$
 denotes the
 (possibly degenerate, for $\sigma =0$)
 Gaussian distribution 
 with mean 0 and variance $\sigma^2$, i.e., 
 for any interval $I \subset \mathbb R$ we have
 \[
 \lim_{n\to \infty} \mu \Big\{
 \frac{S_n (\varphi) - n\langle\mu,\varphi\rangle }{\sqrt{n} } \in I
 \Big\}
 =
 \begin{cases}
 	1 \mbox{ when } I \mbox{ is of the form } I=(-\delta,\delta)  & \mbox{ if } \sigma^2=0,\\
 	\frac{1}{\sqrt{2\pi\sigma^2}}\displaystyle\int_I e^{-t^2 / (2\sigma^2)} dt &  \mbox{ if }
 	\sigma^2 >0.
 \end{cases} \]

   Bj{\"o}rklund and Gorodnik \cite{BG} recently showed that the exponential mixing of all orders implies the CLT. Applying that result, Bianchi and Dinh \cite{BianchiDinh1} deduced that H\"older observables satisfy the CLT for the measure of maximal entropy, which was a long standing question in complex dynamics of Hénon maps (see also \cite{BianchiDinh2} for the case of the measure of maximal entropy of automorphisms of compact Kähler manifolds). Following that strategy, we prove: 
\begin{theorem}[CLT for Hölder and $\mathrm{DSH}^\infty(\P^k)$ observables]\label{tm_CLT} Let $f$ be a generic birational map as above and $\alpha >0$. Take $\varphi \in C^\alpha(\P^k) \cup \mathrm{DSH}^\infty(\P^k)$, then $\varphi$ satisfies the CLT with respect to $\mu$ with 
	\[\sigma^2
	=
	\sum_{n\in \mathbb Z} \langle\mu, \tilde \varphi (\tilde \varphi \circ f^n) \rangle
	=
	\lim_{n\to \infty}
	\frac{1}{n}\int_{X}
	(\tilde \varphi+ \tilde \varphi \circ f + \ldots + \tilde \varphi \circ f^{n-1})^2 d\mu,
	\]
	where $\tilde \varphi := \varphi-\langle\mu,\varphi\rangle$.
\end{theorem}
  To apply \cite{BG} directly, one has to define an algebra of test functions which is stable under iteration for both $f$ and $f^{-1}$. In \cite{BianchiDinh1}, the authors naturally chose for Hénon-Sibony maps the space of functions which are $C^2$ in a neighborhood of $\supp (\mu)$ (for all $n\in \Z$, the indeterminacy set of $f^n$ is disjoint from $\supp (\mu)$ in that case). In our case, the indeterminacy sets may be dense in $\supp (\mu)$, so being $C^2$ in a neighborhood of $\supp(\mu)$ is not stable under iteration. For this reason, in §\ref{Section space}, we define a space of test functions in terms of currents: we ask that $d \varphi \wedge d^c \varphi \leq R$  and $\pm dd^ c \varphi \leq R'$ for some positive closed currents $R, R'$ whose potentials are dominated by the potentials of the Green currents $T^{\pm}$. Note that the condition $d \varphi \wedge d^c \varphi \leq R$ implies in particular that our observables belong to the complex Sobolev space $W^*$ introduced by Dinh and Sibony in \cite{DinhSibonyW} to prove the exponential mixing of order 1 for meromorphic map of large topological degree. In §\ref{Section Preliminary}, we adapt the strategy of \cite{Vigny_decay} to get the necessary estimates for Theorem~\ref{tm_exp}. The space of observables we consider is not stable under iteration but it can be written as the intersection of two spaces, one stable under forward iteration, the other under backward iteration. In §\ref{Section proofs}, we first establish Proposition~\ref{prop-principale}, which is a version of the exponential mixing of all orders where we separate the observables according to the above spaces stable for $f$ and for $f^{-1}$. This allows us to prove Theorem~\ref{tm_exp}, and also to check that the proof of \cite{BG} still applies here, which is the content of Corollary~\ref{lemma 9.1}.  

\section{Background on generic birational maps of $\P^k$ and super-potentials theory} 
We start by recalling basic facts on the super-potentials theory of Dinh and Sibony \cite{superpotentiels} (and the appendix of \cite{Dethelin_Vigny}).
Recall that $\omega$ denotes the Fubini-Study form in $\P^k$. 
Let $S$ be a positive closed $(q,q)$ current on $\P^k$. Its \emph{mass} $\|S\|$ is defined by $\|S\|:= \int_{\P^k} S\wedge \omega^{k-q}$.

Let
$\mathcal{C}_q$ denote the set of positive closed currents of mass $1$ and bidegree $(q,q)$ in
$\P^k$ for $0\leq q \leq k$. For $T \in \mathcal{C}_q$, we consider some quasi-potential $U_T$ of
$T$ (that is $T=\omega^q+dd^cU_T$, see \cite{superpotentiels}[Theorem 2.3.1]). We define the
super-potential $\mathcal{U}_T$ associated to the quasi-potential $U_T$  as the function on
$\mathcal{C}_{k-q+1}$ defined for $S$ smooth by:
$$\mathcal{U}_T(S):=\langle U_T, S\rangle.$$
This definition can be extended by sub-harmonicity along the structural varieties to any
$S \in \mathcal{C}_{k-q+1}$ allowing the value $-\infty$. The \emph{mean} of a super-potential is then
$m_T:=\mathcal{U}_T(\omega^{k-q+1})$ and the super-potential of a current is uniquely defined by
its mean. This is one of the strengths of super-potentials: quasi-potentials differ by a
$dd^c$-closed form whereas super-potentials are
defined up to a constant. We say that a sequence $T_n \in \mathcal{C}_q$ \emph{converges to $T$ in
	the Hartogs' sense}  (or $H$-converges) if $T_n \to T$ in the sense of currents and if we
can choose super-potentials $(\mathcal{U}_{T_n})$ and $\mathcal{U}_{T}$ such that $m_{T_n} \to m_T$ and
$\mathcal{U}_{T_n} \geq \mathcal{U}_T$ for all $n$. In that case, if $S_n \to S$ in the Hartogs' sense in
$\mathcal{C}_{k-q+1}$, then $\mathcal{U}_{T_n}(S_n) \to  \mathcal{U}_{T}(S)$.
A current is said to be more \emph{ $H$-regular} than another one if, choosing suitable means,
its super-potential is greater than the other. All the classical tools (intersection,
pull-back), well defined for smooth forms can be extended as continuous objects for the
Hartogs' convergence. Any such operation, well defined for a given current $T$, is also
well defined for a more $H$-regular current $T'$. More precisely:
\begin{itemize}
	\item We say that $T\in \mathcal{C}_q$ and $S\in \mathcal{C}_{k-q-r}$ are wedgeable if $\mathcal{U}_{T}(S\wedge
	\Omega) >-\infty$ for some smooth $\Omega \in \mathcal{C}_{r+1}$. That condition is symmetric in
	$T$ and $S$ and, if it is satisfied, we can define the wedge product $T \wedge S$ with the
	above continuity property.
	If $T'$ and $S'$ are more $H$-regular than $T$ and $S$ then $T'$ and $S'$ are wedgeable and
	$T' \wedge S'$ is more $H$-regular than $T \wedge S$.
	If $R\in \mathcal{C}_r$, $S\in \mathcal{C}_s$ and $T\in \mathcal{C}_t$ ($r+s+t \leq k$) are such that $R$ and $S$
	are wedgeable and $R\wedge S$ and $T$ are wedgeable then the wedge product $R \wedge S
	\wedge T$ is well defined and that property is symmetric in $R$, $S$ and $T$.
	\item Similarly, for $T \in \mathcal{C}_q$, we say that $T$ is $f^*$-admissible if its
	super-potential is finite at some current of the form $d_{q+1}^{-1}f_*(S)$ for $S\in
	\mathcal{C}_{k-q+1}$ smooth near $I(f)$ ($d_{q+1}$ is the normalization so that 
	$d_{q+1}^{-1}f_*(S)$ is of mass $1$). For such current, we can define its pull-back with
	the above continuity property. If $T'$ is more $H$-regular than $T$ then $T'$ is also 
	$f^*$-admissible and  $d_q^{-1}f^*(T')$ is more $H$-regular than $d_q^{-1}f^*(T)$.
\end{itemize}

Let $f$ be a generic birational map of $\P^k$, hence it satisfies \eqref{gooddim}, \eqref{AS} and \eqref{energy_higher2}. We now give the properties of $f$ we will need (see \cite{Dethelin_Vigny}[Theorem~3.2.8, Theorem 3.4.1,  Corollary
3.4.11, Theorem 3.4.13]). For any $j \leq s$ and $j'\leq k-s$, we can define the \emph{Green currents} $T^+_j$ and $T^-_{j'}$ by 
\[ T^+_j := \lim_n d^{-jn} (f^n)^*(\omega^j)  \ \mathrm{and} \ T^-_{j'} := \lim_n \delta^{-j'n} (f^n)_*(\omega^{j'}),\]
where the convergences are in the Hartogs' sense \cite{Dethelin_Vigny}. For $j=j'=1$, we simply write $T^+$ and $T^-$. Furthermore, let $u^+$ (resp. $u^-$) be a negative quasi-potential of $d^{-1} f^*(\omega)$ (resp. $\delta^{-1} f_*(\omega)$), i.e. $d^{-1}f^*(\omega) =\omega + dd^c u^+$  (resp. $\delta^{-1}f_*(\omega) =\omega + dd^c u^-$). Then, $u^+_\infty:=\sum_n d^{-n}u^+ \circ f^n $ and $u^-_\infty:= \sum_n \delta^{-n}u^- \circ f^{-n}$ define negative quasi-potentials of $T^+$ and $T^-$ (e.g. \cite{Dethelin_Vigny}). 
 Then, $T^+_s\in \mathcal{C}_s$ is equal to $(T^+)^s$ and is invariant $f^*T^+_s=d^sT^+_s$ (similarly
$T^-_{k-s}=(T^-)^{k-s} \in \mathcal{C}_{k-s}$ satisfies $f_*T^-_{k-s}=\delta^{k-s}T^-_{k-s}$).
Finally, $T^+_s$ and $T^-_{k-s}$ are wedgeable and the intersection $\mu:=T_s^+ \wedge
T^-_{k-s}$ is an invariant probability measure that we call the Green measure. We have that $\mu $ integrates the
quasi-potentials of $T^+$ and $T^-$ (and thus $\log \mathrm{dist}(x, I^+)$ and $\log
\mathrm{dist}(x, I^-)$). In the formalism of super-potentials, this means that
$\mathcal{U}_{T^+}(T^+_s \wedge T^-_{k-s})>-\infty$.  The above arguments on Hartogs' regularity
imply that for any $S\in \mathcal{C}_1$ more $H$-regular than $T^+$ (in particular for $d^{-n}
(f^n)^*(\omega)$, $\omega$ and the current $T_n$ defined below) then $S$ and $T^+_s$ are
wedgeable. In that case, $S \wedge T^+_s$ is $(f^n)^*$-admissible and, for any $Q$ more
$H$-regular than $T^-_{k-s}$, $\mathcal{U}_S( T^+_s\wedge Q)$ is finite.
The measure $\mu$ gives no mass to pluripolar sets (hence to proper analytic sets). \\

We shall need the following quantified version of $H$-regularity. 
  \begin{definition}
	Let $a\in \R^+$. We say that $S \in \mathcal{C}_1$ is \emph{more $(H,a)$-regular} than $T^+$ (resp. $T^-$) if there exists a quasi-potential $U_S$ of $S$ such that $0\geq U_S \geq u^+_\infty-a$ (resp. $0\geq U_S \geq u^-_\infty-a$). 
\end{definition} 
Take $S\in \mathcal{C}_1$, a $C^2$ current, then, provided $a$ is large enough, $S$ is more $(H,a)$-regular than $T^+$ and $T^-$.

\section{A space of test functions}\label{Section space}
 Let us recall the definition and basics properties of DSH functions   (\cite{dinhsibony2}). A measurable function $\varphi$ is in $\mathrm{DSH}(\P^k)$ if it can be written as $u_1-u_2$ where $u_i$ are qpsh functions. In particular, $dd^c \varphi = T_1- T_2$ where $T_i$ are positive closed currents. Similarly, if a function $\varphi\in L^1_{\mathrm{loc}}$ satisfies $dd^c\varphi= T_1-T_2$ for such currents, it admits a representative $\tilde{\varphi}$ (i.e. $\varphi \stackrel{a.e.}{=}\tilde{\varphi}$) such that $\tilde{\varphi}\in \mathrm{DSH}(\P^k)$. In particular, for all $x$ outside a pluripolar set of $\P^k$, by the submean inequality for psh functions and upper-semi continuity of such functions, we have that $\tilde{\varphi}(x)$ is determined by
 \begin{equation}\label{eq_mean1} \lim_{r\to 0} \frac{1}{|B(x,r)|}\int_{B(x,r)} |\tilde{\varphi}(y)-\tilde{\varphi}(x)| d\lambda(y)=0. \end{equation}
where $B(x,r)$ is the ball of center $x$ and radius $r$, $\lambda$ is the Lebesgue measure and $|B(x,r)|=\lambda(B(x,r))$ (we are working in a chart but this property does not depend on the choice of such chart).
 As we can see, an important feature of DSH functions is that they are well defined outside a pluripolar set.
\begin{definition}\label{def_A}
 Let $a \in \R^+$.  Consider
	\begin{align*}
	\mathcal{A}_+^a:= \Big\{ & \varphi \in L^\infty(\P^k) \cap W^{1,2}(\P^k) \cap \mathrm{DSH}(\P^k), \mathrm{and}\\ 
	 & a) \ \exists R_1^+\in \mathcal{C}_1, \ \mathrm{more} \ (H,a)\mathrm{-regular\ than} \ T^+,  \exists C^+_1 \geq 0, \ d \varphi \wedge d^c \varphi \leq C_1^+ R_1^+, \\
&c) \ \exists R_2^+\in \mathcal{C}_1, \ \mathrm{more} \ (H,a)\mathrm{-regular\ than} \ T^+,  \exists C^+_2 \geq 0, \ 0 \leq C_2^+ R_2^+\pm dd^c \varphi\Big\}.
	\end{align*}
		\begin{align*}
		\mathcal{A}_-^a:= \Big\{ & \varphi \in L^\infty(\P^k) \cap W^{1,2}(\P^k) \cap \mathrm{DSH}(\P^k), \mathrm{and}\\ 
		&b) \ \exists R_1^-\in \mathcal{C}_1, \ \mathrm{more} \ (H,a)\mathrm{-regular\ than} \  T^-,  \exists C^-_1 \geq 0, \ d \varphi \wedge d^c \varphi \leq C_1^- R_1^-,\\
		&d) \ \exists R_2^-\in \mathcal{C}_1, \ \mathrm{more} \ (H,a)\mathrm{-regular\ than} \  T^-,  \exists C^-_2 \geq 0, \  0\leq C_2^- R_2^- \pm dd^c \varphi.	\Big\}
	\end{align*}
	and $\mathcal{A}^a:= 	\mathcal{A}_+^a \cap 	\mathcal{A}_-^a$.
\end{definition} 
\begin{remark}\label{rm_outside} \normalfont Let $\varphi \in \mathcal{A}_\pm^a$. Since $T^\pm$ give no mass to pluripolar sets, if $R$ is a positive closed $(1,1)$ current which is more $H$-regular than $T^\pm$, then $R$  gives no mass to any algebraic set $Z$ (if not, $T^\pm$ would have uniformly $>0$ Lelong numbers along $Z$ so it would also give a positive mass to $Z$ which we know not to hold). As a consequence, let us show that it is enough to ask for all the above inequalities to hold outside a strict algebraic set. 
	
	First, as $\varphi \in W^{1,2}(\P^k)$, $d\varphi \wedge d^c \varphi$ does not charge sets of zero Lebesgue measure (hence strict algebraic sets).
	 
	  Similarly, assume for example that $0 \leq C_2^+ R_2^++ dd^c \varphi$ only holds outside an algebraic set $Z$, we would have, on a neighborhood of a point $p \in Z$ that $ C_2^+ R_2^+ = dd^c \psi$ for some psh function $\psi$ so that $\psi + \varphi$ coincides with a psh function $\Psi$ outside a pluripolar set of zero Lebesgue measure (containing $Z$). Regularizing by a radial approximation of the unity, we see that $\Psi$ is bounded from above as both $\varphi$ and $\psi$ are so it extends to a psh function on the above neighborhood of $p$  \cite[Chapter I, Theorem 5.24]{Demailly}. As this extension is also bounded from  below by a local potential of $ T^+$ ($\varphi$ is bounded and $\psi$ is bounded from below by a local potential of $ T^+$), its $dd^c$ gives no mass to $Z$. We deduce that $ C_2^+ R_2^++ dd^c \varphi$ actually defines a positive closed current on the whole $\P^k$ which gives no mass to $Z$. In particular, $ 0\leq C_2^+ R_2^++ dd^c \varphi$ holds on $\P^k$.   
\end{remark}

\begin{remark} \normalfont Note that a function $\varphi$ satisfying condition c) in $\mathcal{A}_+^a$ admits a DSH representative since $dd^c \varphi= C^+_2 R_2^+ + dd^c \varphi -  C^+_2 R_2^+$. So, taking  $\varphi \in L^\infty(\P^k) \cap W^{1,2}(\P^k)$ satisfying a) and c) allows to define an element in $\mathcal{A}_+^a$ by taking such a representative (and similarly for $\mathcal{A}_-^a$). In what follows, we check that such choice is compatible with the different operations we consider (for example, it is clear that the sum of representatives is the representative of the sum). 
\end{remark}

\begin{remark}\label{defined_up_to_a_pp} \normalfont Take $\varphi \in \mathcal{A}^a_+$ and consider a probability measure $\nu$ which gives no mass to pluripolar sets (for example, the measure $\mu$ defined above). As $\nu$ does not charge pluripolar set, $\langle \nu, \varphi\rangle$ is well defined and by \eqref{eq_mean1}, we have that for $\nu$-a.e. $x$, $|\varphi(x)|\leq \| \varphi\|_\infty$. In other words, $\|\varphi\|_{\infty, \nu} \leq \|\varphi\|_{\infty}$. 
\end{remark}

\begin{definition}
	 Let $\varphi \in \mathcal{A}_+^a$ and $\psi \in \mathcal{A}_-^a$. We denote:
	 \begin{equation*}\label{def_norme_pm}
	 	\|\varphi\|_{\mathcal{A}_+^a}:= \| \varphi\|_\infty + \sqrt{|\varphi|_1^+ } +|\varphi|_2^+ \ \mathrm{and} \ 	\|\psi\|_{\mathcal{A}_-^a}:= \| \psi\|_\infty +  \sqrt{|\psi|_1^- } +|\psi|_2^- 
	 \end{equation*}
	 where \begin{itemize}
	 	\item $|\varphi|_1^+:= \inf\{ C_1^+, \  C_1^+ \ \mathrm{satisfies} \ a) \}= \min\{ C_1^+, \  C_1^+ \ \mathrm{satisfies} \ a) \}$;
	 	\item $|\psi|_1^-:= \inf\{ C_1^-, \  C_1^- \ \mathrm{satisfies} \ b) \}= \min\{ C_1^-, \  C_1^- \ \mathrm{satisfies} \ b) \}$;
	 	\item $|\varphi|_2^+:= \inf\{ C_2^+, \  C_2^+ \ \mathrm{satisfies} \ c) \}= \min\{ C_2^+, \  C_2^+ \ \mathrm{satisfies} \ c) \}$;
	 	\item $|\psi|_2^-:= \inf\{ C_2^-, \  C_2^- \ \mathrm{satisfies} \ d) \}= \min\{ C_2^-, \  C_2^- \ \mathrm{satisfies} \ d) \}$.
	 \end{itemize}
	 For $\varphi \in \mathcal{A}^a$, we denote 
	 \begin{equation*}\label{def_norme_pm2}
		\|\varphi\|_{\mathcal{A}^a}:= \| \varphi\|_\infty + \sqrt{|\varphi|_1^+ } + \sqrt{|\varphi|_1^- }+|\varphi|_2^+ +|\varphi|_2^- 
	\end{equation*}

	 \end{definition}
\begin{remark}\normalfont The above infima are indeed minima by taking limits in the sense of currents.  
\end{remark}
Observe that $\mathcal{A}_\pm^a \subset \mathcal{A}_\pm^{a'}$ with $\|\varphi\|_{\mathcal{A}^a} \geq \|\varphi\|_{\mathcal{A}^{a'}}$ for $a \leq a'$. The following is essentially classical (\cite{superpotentiels}) and follows from an approximation of unity; we give the proof for the sake of completeness. 
\begin{proposition}\label{prop_regularization}
	Let $a< a'$. Take  $\varphi \in \mathcal{A}_\pm^a$ (resp. $\varphi \in \mathcal{A}^a$), then there exists a sequence $\varphi_n$ of smooth functions such that
	\begin{enumerate}
		\item $\varphi_n \to \varphi$ in $W^{1,2}$ and $\varphi_n(x) \to \varphi(x)$ for all $x$ outside a pluripolar set. Furthermore, if $\varphi$ is qpsh, one can choose $(\varepsilon_n)$ a sequence that decreases to $0$ such that $\varphi_n +\varepsilon_n$ decreases to $\varphi$.
		\item $\limsup \|\varphi_n\|_{\mathcal{A}_\pm^{a'}} \leq   \|\varphi\|_{\mathcal{A}_\pm^{a}}$ (resp. $\limsup \|\varphi_n\|_{\mathcal{A}^{a'}} \leq   \|\varphi\|_{\mathcal{A}^{a}}$).  
	\end{enumerate}
\end{proposition}
\begin{proof}  We only treat the case $\varphi \in \mathcal{A}_+^a$. Let $U$ be a chart centered at $0$ which corresponds to the identity in $\mathrm{PSL}_{k+1}(\C)$, the group of automorphisms of $\P^k$.	
	Let $\begin{cases}
		\Phi: &U \to B(0,1) \subset \C^k\\
	          &g \mapsto g.0 	
	\end{cases}$ and define 
	\[\theta_{\varepsilon}(g)= \frac{1}{\varepsilon^{2k}} \alpha\left(\frac{\Phi(g)}{\varepsilon}\right)\]
	where $\alpha:B(0,1) \to \R^+$ is a smooth radial function with compact support and $\int_{B(0,1)} \alpha d\lambda = 1$ where $\lambda$ is the Lebesgue measure on  $B(0,1)$. Take $\nu$ such that $\Phi_*\nu= \lambda$. Observe that
	\[\int_{U} \theta_{\varepsilon}(g) d\nu = \int_{B(0,1)}\frac{1}{\varepsilon^{2k}} \alpha\left(\frac{z}{\varepsilon}\right)  d\lambda(z)=1.  \]	
	Let $\Psi  \in L^1(\P^k)$, we denote by 
	\[\Psi_n(x) := \int_{\mathrm{PSL}_{k+1}(\C)} \Psi(g.x)  \theta_{1/n}(g)d\nu(g).\]
It is classical that $\Psi_n$ is smooth and converges to $\Psi$ in $L^1$, and furthermore, if $\Psi \in W^{1,2}$, the convergence holds in $W^{1,2}$. Assume now that $\Psi$ is psh near $x$, then, if $x=0$, one immediately has that $(\Psi_n(0))$ 	decreases to $\Psi(0)$ ($\theta$ is radial); for $x$ arbitrary, let $g_x \in  \mathrm{PSL}_{k+1}(\C)$ such that $g_x(0)=x$. Then, $\widetilde{\Psi}:= g_x^* \Psi $ is psh in a neighborhood of $0$ since $g$ is holomorphic; in particular, $(\widetilde{\Psi}_n(0))_n= (\Psi_n(x))_n$ decreases to $\widetilde{\Psi}(0)=\Psi(x) $. As elements of $\mathcal{A}_+^a$ are the difference of two qpsh functions, this implies $\varphi_n(x) \to \varphi(x)$ for all $x$ outside a pluripolar set. For the second part of the first point of the Lemma: if $\varphi$ is qpsh, it is locally the difference of a psh function (hence the regularization decrease) and a smooth function $h$ (for which the convergence is uniform). Hence for $\varepsilon_n:= \|h-h_n\|_\infty$, $\varphi_n +\varepsilon_n$ decreases to $\varphi$ up to extracting a subsequence and replacing $\varepsilon_n$ with a slightly larger number.
	
	Now, if $R^+_1 \in \mathcal{C}_1$ is more $(H,a)$-regular than $T^+$  with 
	\[d \varphi \wedge d^c \varphi \leq |\varphi|^+_1 R_1^+\]
	then by Fubini and Lemma~\ref{CS} below:
	\begin{align*}d \varphi_n \wedge d^c \varphi_n &= \int_{\mathrm{PSL}_{k+1}(\C) \times \mathrm{PSL}_{k+1}(\C)}  d \varphi(gx) \wedge d^c \varphi(g'x) \theta_{1/n}(g)\theta_{1/n}(g')d\nu \otimes\nu(g,g') \\
		&\leq 1/2\int_{\mathrm{PSL}_{k+1}(\C) \times \mathrm{PSL}_{k+1}(\C)}  \Big(d \varphi(gx) \wedge d^c \varphi(gx) \\
		&\quad \quad \quad \quad +d \varphi(g'x) \wedge d^c \varphi(g'x) \Big)\theta_{1/n}(g)\theta_{1/n}(g') d\nu \otimes\nu(g,g') \\
		&\leq \int_{\mathrm{PSL}_{k+1}(\C)}  g^*(d \varphi(x) \wedge d^c \varphi(x))\theta_{1/n}(g)d\nu(g) \\
		&\leq |\varphi|^+_1  \int_{\mathrm{PSL}_{k+1}(\C)}  g^*R_1^+\theta_{1/n}(g)d\nu(g) \\
		&= |\varphi|^+_1  dd^c\int_{\mathrm{PSL}_{k+1}(\C)}  U_{R_1^+}( g(x))\theta_{1/n}(g)d\nu(g)+ |\varphi|^+_1  \int_{\mathrm{PSL}_{k+1}(\C)}  g^*\omega \theta_{1/n}(g)d\nu(g)
	\end{align*}	
	where $U_{R_1^+}$ is a quasi-potential of $R_1^+$ satisfying  $0\geq U_{R_1^+} \geq u^+_\infty-a$. 
	It is classical that $\int_{\mathrm{PSL}_{k+1}(\C)}  g^*\omega \theta_{1/n}(g)d\nu(g)$ converges to $\omega$ in $C^2$-norm, so we can write it as $\omega + dd^cu_n$ where $\|u_n\|_\infty <\varepsilon$ for $n$ large enough. In a chart of a finite atlas of $\P^k$, we can write $U_{R_1^+}= u_{R_1^+}- \psi$ where $u_{R_1^+}$ is psh and $\psi$ is smooth: $\psi$ is a potential of $\omega$ in the chart and $ u_{R_1^+}$ is a potential of $R_1^+$.  With the above notations, $\int_{\mathrm{PSL}_{k+1}(\C)}  U_{R_1^+}( g(x))\theta_{1/n}(g)d\nu(g)$ is non positive (as an average of non positive functions) and it can be written as $(u_{R_1^+})_n- \psi_n$. We saw that $(u_{R_1^+})_n$ decreases to $u_{R_1^+}$ and $\psi_n$ converges uniformly to $\psi$. As $U_{R_1^+} \geq u^+_\infty-a$, we deduce that 
	\begin{align*} 
	0\geq (U_{R_1^+})_n + u_n- \varepsilon &=  (u_{R_1^+})_n- \psi_n +u_n -\varepsilon\\
	                                       &\geq U_{R_1^+} -  \varepsilon'-2\varepsilon \geq u^+_\infty-a'
	\end{align*} 
	by taking $\varepsilon'$ and $\varepsilon$ small enough. This shows that $\varphi_n $ satisfies condition a) in the definition of $\mathcal{A}_+^{a'}$ with $|\varphi_n|_1^+ \leq |\varphi|_1^+$. For condition c), observe that if $R^+_2 \in \mathcal{C}_1$ is more $(H,a)$-regular than $T^+$ with 
	\[0\leq  \pm dd^c \varphi + |\varphi|_2^+ R_2^+\]
	then 
	\[0\leq  \pm dd^c \int_{\mathrm{PSL}_{k+1}(\C)} \varphi(gx) \theta_{1/n}(g)d\nu(g) +  |\varphi|_2^+   \int_{\mathrm{PSL}_{k+1}(\C)} g^* R_2^+ \theta_{1/n}(g)d\nu(g)\]
	so 
	\[0\leq  \pm dd^c \varphi_n  +  |\varphi|_2^+   \int_{\mathrm{PSL}_{k+1}(\C)} g^* R_2^+ \theta_{1/n}(g)d\nu(g)\]
	and we just proved that $\int_{\mathrm{PSL}_{k+1}(\C)} g^* R_2^+ \theta_{1/n}(g)d\nu(g) \in \mathcal{C}_1$ is more $(H,a')$-regular than $T^+$ so again, $\varphi_n$ satisfies condition c) with  $|\varphi_n|_2^+ \leq |\varphi|_2^+$. The proposition follows. 
\end{proof}
\begin{remark}\label{rm_regularization}
	\normalfont It is also possible to approximate $\varphi\in \mathcal{A}_+^a$ by smooth functions in $\mathcal{A}_+^a$ by dividing $\varphi$ by $\rho>1$ since if 
	\[ d\left(\frac{\varphi}{\rho} \right)\wedge d^c \left(\frac{\varphi}{\rho} \right)  \leq \frac{C_1}{\rho^2}(\omega + dd^c \psi) \leq C_1(\omega + dd^c \frac{\psi}{\rho^2}),\] 
with similar arguments for condition c) and d). Observe that then, taking a suitable $\rho_n\to 1$ and defining $\tilde{\varphi}_n = \varphi_n /\rho_n$, we can have that $\lim \|\tilde{\varphi}_n\|_{\mathcal{A}^a_\pm} =   \|\varphi\|_{\mathcal{A}^a_\pm}$ (Proposition~\ref{prop_regularization} ensures the inequality $\leq$ and the other one is immediate by weak convergence) and  $\rho_n \tilde{\varphi}_n+\varepsilon_n$ decreases to $\varphi$ .
\end{remark}
 
From now on, we fix $a$ and we simply denote $\mathcal{A}_\pm$ and $\mathcal{A}$ instead of $\mathcal{A}_\pm ^a$ and $\mathcal{A}^a$ if no confusion can arise.

\begin{theorem} We have that $(\mathcal{A}_\pm, \|.\|_{\mathcal{A}_\pm})$ is a Banach subalgebra of $L^\infty$. Furthermore,  for $\varphi_1,\varphi_2 \in \mathcal{A}_\pm$,  $\|\varphi_1\varphi_2\|_{\mathcal{A}_\pm} \leq 20 \|\varphi_1\|_{\mathcal{A}_\pm} \|\varphi_2\|_{\mathcal{A}_\pm}$. 
	
	Consequently, $(\mathcal{A}, \|.\|_{\mathcal{A}})$ is a Banach subalgebra of $L^\infty$ and for $\varphi_1,\varphi_2 \in \mathcal{A}$, we have $\|\varphi_1\varphi_2\|_{\mathcal{A}} \leq 40 \|\varphi_1\|_{\mathcal{A}} \|\varphi_2\|_{\mathcal{A}}$. 
	\end{theorem}
In the following, we will write $R\leq S$ for two $(1,1)$ currents in the case where only $S$ is positive; this has to be understood as $0\leq S-R $ in the sense of currents.
\begin{lemma}\label{CS} 
		Let $\varphi_1,\varphi_2\in W^{1,2}$. Then, for $c>0$
	\begin{align*}
		c d \varphi_1 \wedge d^c \varphi_1 +c^{-1} d \varphi_2 \wedge d^c \varphi_2 \geq d \varphi_1 \wedge d^c \varphi_2+d \varphi_2 \wedge d^c \varphi_1\\
			c d \varphi_1 \wedge d^c \varphi_1 +c^{-1} d \varphi_2 \wedge d^c \varphi_2 \geq -d \varphi_1 \wedge d^c \varphi_2-d \varphi_2 \wedge d^c \varphi_1.
	\end{align*} 
\end{lemma}
\begin{proof}Let $c'= \sqrt{c}$. Then
\begin{align*}
d  (c'\varphi_1 - c'^{-1} \varphi_2)\wedge d^c(c'\varphi_1 - c'^{-1} \varphi_2) \geq 0 \\ 
 d  (c'\varphi_1 + c'^{-1} \varphi_2)\wedge d^c(c'\varphi_1 + c'^{-1} \varphi_2) \geq 0
\end{align*} Developing  gives the two inequalities of the lemma. 
\end{proof}

\begin{proof}[Proof of the theorem] 
	We only consider the case  $(\mathcal{A}_+, \|.\|_{\mathcal{A}_+})$, $(\mathcal{A}_-, \|.\|_{\mathcal{A}_-})$ is similar and the result for $(\mathcal{A}, \|.\|_{\mathcal{A}})$ immediately follows.

First, $L^\infty(\P^k) \cap W^{1,2}(\P^k) \cap \mathrm{DSH}(\P^k)$ is a vector space on which $\|.\|_\infty$ is a norm; stability under multiplication by a scalar and $\|\lambda \varphi\|_{\mathcal{A}_+}=|\lambda|\|\varphi\|_{\mathcal{A}_+}$ are straightforward.
\subsection*{$\sqrt{|\varphi|^+_1}$ satisfies the triangular inequality.}
We follow ideas of \cite{Vigny_JFA}. Take $\varphi_1, \varphi_2 \in \mathcal{A}_+$, we can assume that 
$|\varphi_1|^+_1 \neq 0$ and $|\varphi_2|^+_1 \neq 0$ (if not, one of these functions is constant and the result is clear). In particular,
\[ d\varphi_i \wedge d^c \varphi_i \leq |\varphi_i|_1^+ R_i\]
for $i=1,2$ where $R_i \in \mathcal{C}_1$ is more $(H,a)$-regular than $T^+$. By applying the lemma~\ref{CS} to  $c>0$ with $c=\sqrt{|\varphi_2|_1^+/|\varphi_1|_1^+}$, we deduce
\begin{align*} d(\varphi_1+\varphi_2) \wedge d^c (\varphi_1+\varphi_2)&=  d\varphi_1 \wedge d^c \varphi_1+ d\varphi_2 \wedge d^c \varphi_2 + d \varphi_1 \wedge d^c \varphi_2+d \varphi_2 \wedge d^c \varphi_1\\
	& \leq |\varphi_1|_1^+ R_1+ |\varphi_2|_1^+ R_2 + 	c d \varphi_1 \wedge d^c \varphi_1 +c^{-1} d \varphi_2 \wedge d^c \varphi_2\\
	&\leq  |\varphi_1|_1^+(1+c) R_1+ |\varphi_2|_1^+(1+c^{-1}) R_2\\
	&\leq  (|\varphi_1|_1^++\sqrt{|\varphi_1|_1^+|\varphi_2|_1^+})R_1+ (|\varphi_2|_1^+ +\sqrt{|\varphi_1|_1^+|\varphi_2|_1^+})R_2.
\end{align*}
Let $R':= (|\varphi_1|_1^++\sqrt{|\varphi_1|_1^+|\varphi_2|_1^+})R_1+ (|\varphi_2|_1^+ +\sqrt{|\varphi_1|_1^+|\varphi_2|_1^+})R_2$. Then $R'$ has mass 
$|\varphi_1|_1^++2\sqrt{|\varphi_1|_1^+|\varphi_2|_1^+}+ |\varphi_2|_1^+ = \left(\sqrt{|\varphi_1|_1^+}+\sqrt{|\varphi_2|_1^+}\right)^2$. Furthermore, if $0\geq U_{R_i} \geq u^+_\infty -a$ are  quasi-potentials of $R_i$ for $i=1,2$, then 
\begin{align*} 
R' &= (|\varphi_1|_1^++\sqrt{|\varphi_1|_1^+|\varphi_2|_1^+})R_1+ (|\varphi_2|_1^+ +\sqrt{|\varphi_1|_1^+|\varphi_2|_1^+})R_2\\
& =\|R'\| \left(\omega + dd^c \left(\frac{|\varphi_1|_1^++\sqrt{|\varphi_1|_1^+|\varphi_2|_1^+}}{\|R'\|}U_{R_1} + \frac{|\varphi_2|_1^++\sqrt{|\varphi_1|_1^+|\varphi_2|_1^+}}{\|R'\|} U_{R_2}\right) \right).  \end{align*}
In particular, the current $R:= \omega + dd^c \left(\frac{|\varphi_1|_1^++\sqrt{|\varphi_1|_1^+|\varphi_2|_1^+}}{\|R'\|}U_{R_1} + \frac{|\varphi_2|_1^++\sqrt{|\varphi_1|_1^+|\varphi_2|_1^+}}{\|R'\|} U_{R_2}\right)$ is in $\mathcal{C}_1$ admits the quasi-potential $\frac{|\varphi_1|_1^++\sqrt{|\varphi_1|_1^+|\varphi_2|_1^+}}{\|R'\|}U_{R_1} + \frac{|\varphi_2|_1^++\sqrt{|\varphi_1|_1^+|\varphi_2|_1^+}}{\|R'\|} U_{R_2}$ which satisfies:
\begin{align*}
	0\geq  \frac{|\varphi_1|_1^++\sqrt{|\varphi_1|_1^+|\varphi_2|_1^+}}{\|R'\|}U_{R_1} + \frac{|\varphi_2|_1^++\sqrt{|\varphi_1|_1^+|\varphi_2|_1^+}}{\|R'\|} U_{R_2}\geq u_\infty^+-a.
	\end{align*}
In particular, $d(\varphi_1+\varphi_2) \wedge d^c (\varphi_1+\varphi_2)$ is indeed bounded by $\left(\sqrt{|\varphi_1|_1^+}+\sqrt{|\varphi_1|_1^+}\right)^2R$ where $R \in \mathcal{C}_1$ is more $(H,a)$-regular than $ T^+$. Thus $\varphi_1+\varphi_2$ satisfies condition a) and $|\varphi_1+\varphi_2|^+_1 \leq \left(\sqrt{|\varphi_1|_1^+}+\sqrt{|\varphi_2|_1^+}\right)^2$, which gives the triangular inequality for $\sqrt{|\cdot|_1^+}$ . 
\subsection*{$|\varphi|^+_2$ satisfies the triangular inequality.} Again, take $\varphi_i \in \mathcal{A}_+$ for $i= 1,2$. Let $R_i\in  \mathcal{C}_1$ be more $(H,a)$-regular than $ T^+$  with $ \pm dd^c \varphi_i + |\varphi_i|_2^+ R_i \geq 0$ for $i= 1,2$. Write $R_i= \omega + dd^c U_{R_i}$ with $0\geq U_{R_i} \geq u^+_\infty-a$. One has 
\[ ( |\varphi_1|_2^+ R_1 + |\varphi_2|_2^+ R_2) \pm (dd^c \varphi_1+dd^c \varphi_2) \geq 0 \]
so $R:=  \frac{|\varphi_1|_2^+ R_1 + |\varphi_2|_2^+ R_2}{|\varphi_1|_2^++|\varphi_2|_2^+}\in \mathcal{C}_1$ and one checks as above that $R$ is more $(H,a)$-regular than $T^+$. This implies that $\varphi_1+\varphi_2$ satisfies condition c) and $|\varphi_1+\varphi_2|^+_2\leq |\varphi_1|_2^++|\varphi_2|_2^+$ so $|.|^+_2$ satisfies the triangular inequality.
\subsection*{$\mathcal{A}_+$ is an algebra} Only  the stability under product remains. Let $\varphi_i\in \mathcal{A}_+$, $i=1,2$. Clearly, $\varphi_1\varphi_2 \in L^\infty$. We can assume that $\varphi_1 \neq 0$. We compute
\begin{align*}
	d(\varphi_1\varphi_2)\wedge d^c (\varphi_1\varphi_2)= |\varphi_2|^2 	d\varphi_1\wedge d^c \varphi_1+ |\varphi_1|^2 	d\varphi_2\wedge d^c \varphi_2 +\varphi_1\varphi_2( d\varphi_1\wedge d^c \varphi_2+d\varphi_2\wedge d^c \varphi_1).
\end{align*}
Observe that every term is well defined as the $(1,1)$ forms appearing are $L^1$ and the multiplying functions are bounded.  Choosing $c>0$ and applying Lemma~\ref{CS} give
\begin{align*}
\varphi_1\varphi_2( d\varphi_1\wedge d^c \varphi_2+d\varphi_2\wedge d^c \varphi_1)= &(\varphi_1\varphi_2+ \|\varphi_1 \varphi_2 \|_\infty) ( d\varphi_1\wedge d^c \varphi_2+d\varphi_2\wedge d^c \varphi_1)\\
& -\|\varphi_1 \varphi_2 \|_\infty ( d\varphi_1\wedge d^c \varphi_2+d\varphi_2\wedge d^c \varphi_1) \\
  &\leq (\varphi_1\varphi_2+ \|\varphi_1 \varphi_2 \|_\infty) ( c d\varphi_1\wedge d^c \varphi_1+c^{-1} d\varphi_2\wedge d^c \varphi_2)\\
  & +\|\varphi_1 \varphi_2 \|_\infty ( cd\varphi_1\wedge d^c \varphi_1+c^{-1}d\varphi_2\wedge d^c \varphi_2) \\
  &\leq  3\|\varphi_1\|_\infty \|\varphi_2\|_\infty \left(c |\varphi_1|_1^+ R^+_1 +c^{-1}|\varphi_2|_1^+ R^+_2 \right) 
\end{align*}
 where $R^+_i\in \mathcal{C}_1$ are currents more $(H,a)$-regular than $ T^+$ with $d\varphi_i\wedge d^c \varphi_i\leq |\varphi_i|_1^+ R^+_i$ for $i=1,2$. Injecting this inequality in $d(\varphi_1\varphi_2)\wedge d^c (\varphi_1\varphi_2)$ gives
 \begin{align*}
 	d(\varphi_1\varphi_2)\wedge d^c (\varphi_1\varphi_2)\leq & (\|\varphi_2\|_\infty^2 + 3c\|\varphi_1\|_\infty \|\varphi_2\|_\infty) |\varphi_1|_1^+ R^+_1+\\
 	&\quad   (\|\varphi_1\|_\infty^2 + 3c^{-1}\|\varphi_1\|_\infty \|\varphi_2\|_\infty) |\varphi_2|_1^+ R^+_2.
 \end{align*}
Taking $c=\| \varphi_2\|_\infty / \| \varphi_1\|_\infty$ gives
 \begin{align*}
	d(\varphi_1\varphi_2)\wedge d^c (\varphi_1\varphi_2)\leq & 4\|\varphi_2\|_\infty^2|\varphi_1|_1^+ R^+_1+4\|\varphi_1\|_\infty^2 |\varphi_2|_1^+ R^+_2
\end{align*}
hence, we deduce as above that $\varphi_1\varphi_2$ satisfies the condition a) of $\mathcal{A}_+$ (and is thus in $W^{1,2}$ by \cite[Proposition 3.1]{DinhSibonyW}) with $|\varphi_1\varphi_2|_1^+ \leq 8 \|\varphi_1\|_{\mathcal{A}_+}^2 \|\varphi_2\|_{\mathcal{A}_+}^2$.

We now check condition c) and the fact that $\varphi_1 \varphi_2$ is indeed DSH since it is the DSH representative of the product. Let $S_i^+ \in \mathcal{C}_1$, more $(H,a)$-regular than $T^+$ such that $\pm dd^c \varphi_i \leq  |\varphi_i|_2^+ S_i^+$  for $i=1,2$.
\begin{align*}
	dd^c(\varphi_1\varphi_2)= \varphi_2 dd^c \varphi_1 + \varphi_1 dd^c\varphi_2+ 	d\varphi_1\wedge d^c \varphi_2+d\varphi_2\wedge d^c \varphi_1.
\end{align*}
As before, we can assume that $|\varphi_1|_2^+ \neq 0$ (if not, $\varphi_1$ is constant and the result is straightforward). Take $c^2:= |\varphi_2|_2^+/|\varphi_1|_2^+$ so Lemma~\ref{CS} gives
\begin{align*}
\pm (d\varphi_1\wedge d^c \varphi_2+d\varphi_2\wedge d^c \varphi_1) &\leq c d\varphi_1\wedge d^c \varphi_1+ c^{-1} d\varphi_2\wedge d^c \varphi_2 \leq c |\varphi_1|_1^+R^+_1 + c^{-1}|\varphi_2|_1^+R^+_2 \\
& \leq \sqrt{|\varphi_1|_1^+|\varphi_2|_1^+ }( R^+_1 +R^+_2).
\end{align*}
On the other hand,
\[\pm \varphi_2 dd^c \varphi_1 =\pm\left( (\varphi_2 + \|\varphi_2\|_\infty ) dd^c \varphi_1 - \|\varphi_2\|_\infty dd^c \varphi_1 \right)   
\leq 3 \|\varphi_2\|_\infty |\varphi_1|_2^+ S_1^+ \]
and similarly $\pm \varphi_1 dd^c \varphi_2 
\leq 3 \|\varphi_1\|_\infty |\varphi_2|_2^+ S_2^+$. Combining this and taking minimal currents for the properties a) and c) imply 
\begin{align*}
	\pm dd^c(\varphi_1\varphi_2)\leq 3 \|\varphi_2\|_\infty \|\varphi_1\|_2^+ S_1^+ + 3 \|\varphi_1\|_\infty \|\varphi_2\|_2^+ S_2^++  \sqrt{|\varphi_1|_1^+|\varphi_2|_1^+ }( R^+_1 +R^+_2)
\end{align*}
which implies that $\varphi_1\varphi_2$ satisfies condition c) with $| \varphi_1\varphi_2 |_2^+ \leq 8 \|\varphi_1\|_{\mathcal{A}_+} \|\varphi_2\|_{\mathcal{A}_+}$. In particular, $\varphi_1\varphi_2 $ coincides almost everywhere with a DSH function so it remains to check that $\varphi_1\varphi_2$ is DSH. For that, let $E\subset \P^k$ be a pluripolar set such that 
\begin{equation}
	\label{eq_conv_pp}
	\forall x \neq E, \forall i \in \{1,2\}, \  \lim_{r\to 0} \frac{1}{|B(x,r)|}\int_{B(x,r)} |\varphi_i(y)-\varphi_i(x)| d\lambda(y)=0, \end{equation}  
where again we work locally. Then, for such $x\notin E$ and $r >0$, 
\begin{align*}
	\int_{B(x,r)} |(\varphi_1\varphi_2)(y)-(\varphi_1\varphi_2)(x)| d\lambda(y)\leq & 
\int_{B(x,r)} |\varphi_1(y)\varphi_2(y)-\varphi_1(y)\varphi_2(x)| d\lambda(y)\\
	 &+ \int_{B(x,r)} |\varphi_1(y)\varphi_2(x)-\varphi_1(x)\varphi_2(x)| d\lambda(y) \\
	 \leq & \|\varphi_1\|_\infty  \int_{B(x,r)} |\varphi_2(y)-\varphi_2(x)| d\lambda(y)\\
	 &+ |\varphi_2(x)|\int_{B(x,r)} |\varphi_1(y)-\varphi_1(x)| d\lambda(y). 
\end{align*}
Dividing by $|B(x,r)|$, using $\|\varphi_1\|_\infty<\infty$, $|\varphi_2(x)| <\infty$ and  \eqref{eq_conv_pp} for $i=1,2$, and letting $r\to 0$ imply that for all $x \notin E$, $(\varphi_1 \varphi_2)(x)$ is a DSH representative of $\varphi_1 \varphi_2$ since the right-hand side goes to 0.

Combining all the above inequalities gives that $\mathcal{A}$ is indeed an algebra and we have the bound
\[ \|\varphi_1 \varphi_2\|_{\mathcal{A}_+}\leq 20  \|\varphi_1 \|_{\mathcal{A}_+} \| \varphi_2\|_{\mathcal{A}_+}.\]

\subsection*{$\mathcal{A}_+$ is complete.} The proof is similar to \cite[Proposition 1]{Vigny_JFA}; the key point being that the limit of a converging sequence of currents which are more $(H,a)$-regular that $T^+$ is obviously again more $(H,a)$-regular that $T^+$. 
\end{proof}

\begin{proposition}\label{prop_stability}
	\begin{enumerate}
		\item  Let $\varphi\in \mathcal{A}_+$, then $f^*\varphi\in \mathcal{A}_+ $ with $|f^* \varphi|_1^+ \leq d  |\varphi|_1^+$ and $|f^* \varphi|_2^+ \leq d  |\varphi|_2^+$. In particular, $\|f^* \varphi \|_{\mathcal{A}^+} \leq d \|\varphi \|_{\mathcal{A}^+}$. 
		\item Let $\varphi\in \mathcal{A}_-$, then $f_*\varphi \in\mathcal{A}_- $  with $|f_* \varphi|_1^- \leq \delta  |\varphi|_1^-$ and $|f_* \varphi|_2^- \leq \delta  |\varphi|_2^-$. In particular, $\|f_* \varphi \|_{\mathcal{A}^-} \leq \delta \|\varphi \|_{\mathcal{A}^-}$. 
	\end{enumerate} 
\end{proposition}
\begin{proof}
	By \cite{Dinh_Sibony_distrib}, the pull-back $f^* \varphi$ of a DSH function is again DSH.  We only prove $|f^* \varphi|_1^+ \leq d  |\varphi|_1^+$, the other bounds are similar. Since $\varphi \in \mathcal{A}_+$, then 
	\[d\varphi\wedge d^c\varphi \leq |\varphi|_1^+ R^+_1\]
where $R^+_1 = \omega + dd^c U_{R^+_1} \in \mathcal{C}_1$ with $0 \geq  U_{R^+_1} \geq u^+_\infty -a$. As $R^+_1$ is more $H$-regular than $T^+$, its pull-back by $f$ is well defined (\cite{Dethelin_Vigny}) hence, we can write
 	\[df^*\varphi\wedge d^cf^*\varphi=f^*(d\varphi\wedge d^c\varphi) \leq f^*(|\varphi|_1^+ R^+_1),\]
 note that the above computations make sense outside $I^+$ where none of the above currents have mass so the inequality is valid on the whole $\P^k$ by Remark~\ref{rm_outside}. Now, $f^*(R^+_1)= f^*(\omega)+  dd^c U_{R^+_1} \circ f$ and recall we denoted  $d^{-1}f^*(\omega) =\omega + dd^c u^+$ and $u^+_\infty=\sum_n d^{-n}u^+ \circ f^n $ where $u^+$ is negative. In particular, we have the relation $u^+_\infty= u^+ + d^{-1}  u^+_\infty \circ f $. 
 
  Hence a quasi-potential of $d^{-1}f^*(R^+_1)\in \mathcal{C}_1$ is $u^++ d^{-1}U_{R^+_1} \circ f$ which is negative and
 \[ u^++ d^{-1}U_{R^+_1} \circ f \geq  u^+ + d^{-1} u^+_\infty \circ f -d^{-1}a \geq u^+_\infty -d^{-1}a\]
 which concludes the proof as $-d^{-1}a \geq -a$.
\end{proof}
\begin{example}\normalfont \label{ex_smooth}
	Let $\varphi$ be a $C^2$ function, then obviously $d\varphi \wedge d^c \varphi \leq \|\varphi\|_{C^2} \omega$ and $0\leq \|\varphi\|_{C^2} \omega \pm dd^c \varphi$ so $\varphi \in \mathcal{A}^\pm$ with $\|\varphi\|_{\mathcal{A}}^{\pm} \leq \|\varphi\|_{C^2}$ (renormalizing the $C^2$-norm if necessary).	
\end{example}

\begin{example}\normalfont  \label{ex_DSH}
The space $\mathrm{DSH}_a^\infty(\P^k)$ was defined in the introduction. We claim that $\mathrm{DSH}_a^\infty(\P^k)\subset \mathcal{A}^{a+1/12}$ with a continuous inclusion. Indeed, it is clear that $\varphi \in  \mathrm{DSH}_a^\infty(\P^k)$ satisfies condition c) and d) with $|\varphi|_2^+ = 2\|T_i\|$ and any $a'\geq a$ ($\pm dd^c \varphi \leq T_1+T_2$). It remains to check a) and b). Observe that
\begin{align*} d\varphi \wedge d^c \varphi &= dd^c\varphi^2/2 - \varphi dd^c \varphi \leq  dd^c\varphi^2/2 + 3\|\varphi\|_\infty (T_1+T_2) \\
	&\leq dd^c\varphi^2/2 + 3\|\varphi\|_\infty \|T_1\| (2 \omega +dd^c u_1 +dd^c u_2) \\
	&\leq 6\|\varphi\|_\infty(\|\varphi\|_\infty+ \|T_1\|)\left(\omega+ dd^c \left( \frac{u_1 +u_2}{2} + \frac{\varphi^2}{12 \|\varphi\|_\infty(\|\varphi\|_\infty+ \|T_1\|)}-\frac{1}{12}\right)\right).
	\end{align*}
As 
\[ 0 \geq \frac{u_1 +u_2}{2} + \frac{\varphi^2}{12 \|\varphi\|_\infty		(\|\varphi\|_\infty+ \|T_1\|)}-\frac{1}{12}\geq -a-\frac{1}{12}+u_\infty^+\]
So $\varphi$ satisfies a) and b) for $a+1/12$ and $|\varphi|_1^\pm \leq 6\|\varphi\|_\infty(\|\varphi\|_\infty+ \|T_1\|)$ so $\varphi\in \mathcal{A}^{a+1/12}$ with $\|\varphi\|_{\mathcal{A}^{a+1/12}}\leq C \|\varphi\|_{\mathrm{DSH}_a^\infty(\P^k)}$ for a constant $C$ independent of $\varphi$. 
\end{example}

\section{Preliminary estimates}\label{Section Preliminary}
 We give some exponential estimates which we will need to prove the mixing of order $p$. While they are similar to \cite{Vigny_decay}, here we work with $\mathcal{A}_\pm$ instead of $C^2$ functions.

We will use the following dynamical cut-off function of \cite{Vigny_decay} whose construction we recall. Let   $h\in C^\infty(\R)$, an increasing map with
 \[\begin{cases} h(x)=0 \ \mathrm{for} \ x \leq -2,\\
 	             h(x)=1 \ \mathrm{for} \ x \geq -1
 \end{cases}
 	\]
 Let $v_n$ be a quasi-potential of
 $d^{-n} (f^*)^{n}(\omega)$. Then $v_n$ is smooth outside $I(f^{n})$ and has singularities
 in $O(\log \mathrm{dist}(x, I(f^{n})))$ (see \cite{Dethelin_Vigny}[Lemma 3.2.4]). Subtracting a
 large enough constant to $v_n$, we can assume that the function $w_n:=-\log(- v_n)$ is
 qpsh and in $W^*$ : it satisfies $T_n:=dd^c w_n+ \omega \geq 0$ and $dw_n\wedge d^c
 w_n\leq T_n$ (\cite{DinhSibonyW}[Proposition 4.7]). Since $w_n \geq v_n$, we
 have that $T_n$ is more $H$-regular than  $d^{-n} (f^*)^{n}(\omega)$.
 Now for $A >0$, we consider the smooth function:
 $$\chi_A := h\left(\frac{w_n}{A}\right).$$
 Then, there exists a constant $C$
 that does not depend on $A$ such that:
 \begin{itemize}
 	\item $d \chi_A \wedge d^c \chi_A = \frac{|h'(w_n/A)|^2}{A^2} d w_n \wedge d^c w_n \leq
 	\frac{C}{A^2} T_n$.
 	\item $0 \leq  \frac{C}{A}(T_n+\omega) \pm dd^c\chi_A $.
 	\item $\chi_A =0$ in a neighborhood of $I(f^{n})$
 	\item $\lim_{A\to \infty} 1-\chi_A \searrow  1_{I(f^{n})}$ where $1_{I(f^{n})}$ is the
 	characteristic function of $I(f^{n})$.
 \end{itemize}
Let $ \varphi \in \mathcal{A}$ ($a$ being fixed) so $(f^n)^*(\varphi)\in \mathcal{A}_+$ by Proposition \ref{prop_stability}. In particular, since $T^+ _s\wedge \omega^{k-s}$ is more $H$-regular than $\mu$  (because $\omega^{k-s}$ is more $H$-regular than $T^-_{k-s}$), it does not charge pluripolar sets so by Remark~\ref{defined_up_to_a_pp}, we can define  
 the quantity $c_n$ and the function $\varphi_n$ by:
\begin{equation}\label{varphi_n}
\begin{cases}
	c_n:=\int_{\P^k} f^*(\varphi_{n-1}) T^+ _s\wedge \omega^{k-s}\\
	\varphi_n:= f^*(\varphi_{n-1})-c_n
\end{cases}
\end{equation}
with $c_0:=\int_{\P^k} \varphi T_s^+ \wedge \omega^{k-s}$ and $\varphi_0:= \varphi-c_0$.
 By
construction, we have:
$$(f^n)^* \varphi = \sum_{i=0}^{n} c_i+ \varphi_n \ \mathrm{and} \  \int_{\P^k} \varphi
\circ f^n T^+_s \wedge \omega^s =\sum_{i=0}^{n} c_i. $$
Hence, the following proposition shows in particular that $T_s^+ \wedge
d^{-sn}f^n_*\omega^{k-s}$ converges to $\mu$ in the sense of currents with exponential
estimates for observables in $\mathcal{A}$.
\begin{proposition}\label{rate_measure}
	There exists a constant $C$, independent of $n$ and $\varphi\in \mathcal{A}$ such that
	\begin{equation}\label{estime1}
		|c_n| \leq C \sqrt{\delta}^{-n} \|\varphi\|_{\mathcal{A}_-}.
	\end{equation}
	Consequently, $ |\mu(\varphi_n)|=|\mu(\varphi) -\sum_0^n c_i|\leq C \sqrt{\delta}^{-n}
	\|\varphi\|_{\mathcal{A}_-}$.
\end{proposition}
We will use the notation $S_\varepsilon$ for a smooth approximation of $\delta^{-1} f_*(\omega)=\omega+dd^c u^-$, more $H$-regular than $\delta^{-1} f_*(\omega)$ and let $u_\varepsilon$ be a quasi-potential of $S_\varepsilon$ that decreases to $u^-$ when $\varepsilon\to 0$. We denote
$R_\varepsilon:=\sum_{i=0}^{k-s-1}  S_\varepsilon^i\wedge \omega^{k-s-1-i}$. Similarly, we let $T^+_{1,\varepsilon'}$ be a smooth approximation of $T^+$ whose quasi-potentials decrease to $u^+_\infty$ when $\varepsilon' \to 0$.  
\begin{proof} By Proposition~\ref{prop_regularization}, it suffices to prove the proposition when $\varphi$ is smooth which we assume from now on. In particular, we follow step by step the proof of \cite[Proposition 1]{Vigny_decay}. \cite[Lemmas 1, 2 and 4]{Vigny_decay} can be applied directly, since they stand for $\varphi$ smooth or were proved by an approximation's argument which applies here. We just need to modify Lemma 3 of \cite{Vigny_decay} and show that 
	\[I_1:=\left|\int_{\P^k}  -\chi_Ad\varphi_{n} \wedge d^c u_\varepsilon \wedge
	(T_{1,\varepsilon'}^+)^s \wedge R_\varepsilon\right| \leq  C\|\varphi\|_{\mathcal{A}_-} \sqrt{\delta}^{-n} \]    
 where $C$ does not depend on $A$, $\varepsilon$, $\varepsilon'$, $n$ or $\varphi$. 
 
 Again, following the proof of \cite[Lemma 3]{Vigny_decay}, we see that to get the bound of the term $I_1$, it suffices to show
\[ \left|\int_{\P^k}  -\chi_Ad\varphi_{n} \wedge d^c u_\varepsilon \wedge
T^+_s \wedge R_\varepsilon\right| \leq  C\|\varphi\|_{\mathcal{A}_-} \sqrt{\delta}^{-n}.  \]
First, $\chi_A d\varphi_{n} \wedge d^c \varphi_n = \chi_A d(f^n)^*\varphi \wedge d^c (f^n)^* \varphi= \chi_A (f^n)^*\left(d \varphi \wedge d^c\varphi \right)$ since every term is smooth on the support of $\chi_A$.  Let $R_1^-\in \mathcal{C}_1$, more $(H,a)$-regular than $T^-$, such that $d \varphi \wedge d^c\varphi \leq |\varphi|_1^- R_1^-$.
Using Cauchy-Schwarz inequality \cite[Lemme 1]{Okada} gives:
	\begin{align*} &\left|\int_{\P^k}  -\chi_Ad\varphi_{n} \wedge d^c u_\varepsilon \wedge
	T^+_s \wedge R_\varepsilon\right|^2 \leq \\
	 &\quad \left|\int_{\P^k}  \chi_A  |\varphi|_1^-  (f^n)^*(R_1^-)\wedge 
	T^+_s \wedge R_\varepsilon \right| \left|\int_{\P^k}  \chi_A  d u_\varepsilon\wedge d^c u_\varepsilon \wedge T^+_s \wedge R_\varepsilon \right|. \end{align*}
Again, the second term of the product of the right-hand side is bounded by the proof of  \cite[Lemma 3]{Vigny_decay} so we just have to control the first term. Since $R_1^-$ is more $H$-regular than $T^-$ and $T^-$ and $ 
T^+_s$ are wedgeable, then $R_1^-$ and $ 
T^+_s$ are wedgeable and $R_1^-\wedge 
T^+_s$ is more $H$-regular than $T^-\wedge T^+_s$. In particular, the pull-back  $(f^n)^*(R_1^-\wedge 
T^+_s)$ is well defined and has mass $d_{s+1}^n$. Hence, we write  $\chi_A    (f^n)^*(R_1^-)\wedge 
T^+_s= \chi_A    d_s^{-n}(f^n)^*(R_1^-\wedge 
T^+_s)$. The current $R_\varepsilon$ has mass $k-s$, in particular, the integral can be bounded cohomologically by 
\[ \left|\int_{\P^k}  \chi_A  |\varphi|_1^-  (f^n)^*(R_1^-)\wedge 
T^+_s \wedge R_\varepsilon \right|\leq  \frac{|\varphi|_1^-(k-s)d_{s+1}^n}{d_s^n} =\frac{|\varphi|_1^-
	(k-s)}{\delta^n}\]
and the bound for $I_1$ follows. 
\end{proof}
In what follows, we denote by $ \mathcal{U}_{T^+_{s}}$ (resp. $\mathcal{U}_{T^-_{k-s}}$) the super-potential of $T_s^+$ (resp. $T^-_{k-s}$) given by the quasi-potential $u_{\infty}^+ \sum_{i=0}^{s-1} (T^+)^i\wedge \omega^{s-1-i}$  (resp. $u_{\infty}^- \sum_{i=0}^{k-s-1} (T^-)^i\wedge \omega^{k-s-1-i}$). 
\begin{lemma}
	There exists a constant $K$, independent of $n$, such that for any $R\in \mathcal{C}_1$, more $(H,a)$-regular than $T^-$, we have
\begin{equation}\label{eqbound_U-}0\leq - \mathcal{U}_{T^-_{k-s}} \left( \frac{1}{d_{s+1}^n} (f^n)^*(R\wedge T^+_{s})\right) \leq  K\sqrt{\delta}^n(1+a) \end{equation}
	and, for any $S\in \mathcal{C}_1$, more $(H,a)$-regular than $T^+$, we have 
\begin{equation}\label{eqbound_U+}0\leq - \mathcal{U}_{T^+_{s}} \left( \frac{1}{d_{s-1}^n} (f^n)_*(S\wedge T^-_{k-s})\right) \leq  K\sqrt{d}^n(1+a). \end{equation}	
\end{lemma}
\begin{proof} By symmetry, we only prove the first inequality that we write in terms of quasi-potentials. Continuity of pull-backs, intersections and evaluations for the Hartogs' convergence makes it enough to prove the lemma for $R$ smooth, which we assume from now on. We follow the proof \cite[Lemma 8]{Vigny_decay} where the result was proved when $R=\omega$.
 
In particular, writing $R= \omega + dd^c U_R$ with $0\geq U_R \geq u^-_\infty -a$, it suffices to show
		\[\left| \int_{\P^k} \frac{1}{d_{s+1}^n} (f^n)^*(dd^cU_R\wedge T^+_s)\wedge U_{T^-_{k-s}}  \right|\leq K \sqrt{\delta}^n(1+a). \]
	We apply Stokes and the invariance (up to regularizing every terms and using the cut-off function $\chi_A$ as we do in more details in Proposition~\ref{prop-principale}) $(f^n)^*(d^cU_R\wedge T^+_s)= d_s^nd^c(f^n)^*(U_R)\wedge T^+_s$ to rewrite it as:
\[\int_{\P^k} -\frac{d_s^n}{d_{s+1}^n} (f^n)^*(d^cU_R)\wedge T^+_s\wedge d U_{T^-_{k-s}}  \]
so by Cauchy-Schwarz inequality \cite[Lemme 1]{Okada}, writing $U_{T^-_{k-s}}=  u_{\infty}^- Q$ with $Q:=\sum_{i=0}^{k-s-1}  (T^-)^i\wedge \omega^{k-s-1-i}$:
\begin{align*}
		&\frac{d_s^{2n}}{d_{s+1}^{2n}} \left|\int_{\P^k} (f^n)^*(d^cU_R)\wedge T^+_s\wedge d U_{T^-_{k-s}}\right|^2 \leq \\
		& \quad\quad  		\frac{d_s^{2n}}{d_{s+1}^{2n}}\left| \int_{\P^k} (f^n)^*(dU_R)\wedge (f^n)^*(d^cU_R)\wedge T^+_s\wedge Q   \right|	\left|\int_{\P^k} du_\infty^- \wedge d^cu_\infty^-\wedge T^+_s\wedge Q\right|.
\end{align*}
In \cite[Lemma 8]{Vigny_decay}, it is shown that the second integral of the product is bounded by some constant $C_1^2$. So by Stokes, we are left with 
\begin{align*}I_2:&= \int_{\P^k} (f^n)^*(dU_R)\wedge (f^n)^*(d^cU_R)\wedge T^+_s\wedge Q= \int_{\P^k} \frac{1}{d_s^n}(f^n)^*(dU_R\wedge d^cU_R\wedge T^+_s)\wedge Q\\
&=   \int_{\P^k} \frac{d^n_{s+1}}{d_s^n} dU_R\wedge d^cU_R\wedge T^+_s\wedge \frac{1}{d_{s+1}^n}(f^n)_*Q\\
&= \int_{\P^k} -\frac{d^n_{s+1}}{d_s^n}U_R dd^cU_R\wedge T^+_s\wedge \frac{1}{d_{s+1}^n}(f^n)_*Q
\end{align*}
where we should have used the cut-off function $\chi_A$ to make the computations rigorous.  Our goal is to get rid of terms depending on $R$ using Stokes and  $0\geq U_R \geq u^-_\infty -a$. Let us denote $Q_n:= d_{s+1}^{-n}(f^n)_*Q$. It is a positive closed current of mass $k-s$ and we are left with $\int_{\P^k} -U_R dd^cU_R\wedge T^+_s\wedge Q_n=\int_{\P^k} -U_R R\wedge T^+_s\wedge Q_n + \int_{\P^k} U_R \omega \wedge T^+_s\wedge Q_n$. The last integral is $\leq 0$ because $U_R \leq 0$ and, using $-U_R \leq -u^-_\infty+a$, we deduce
\[\frac{d_s^{n}}{d_{s+1}^{n}} I_2 \leq \int_{\P^k} -u^-_\infty R\wedge T^+_s\wedge Q_n+  a \int_{\P^k}  R\wedge T^+_s\wedge Q_n\leq \int_{\P^k} -u^-_\infty R\wedge T^+_s\wedge Q_n+  a(k-s).\]
Replacing $R= \omega + dd^c U_R$, we deduce 
\[\frac{d_s^{n}}{d_{s+1}^{n}} I_2 \leq\int_{\P^k} -u^-_\infty \omega\wedge T^+_s\wedge Q_n + \int_{\P^k} -u^-_\infty dd^c U_R\wedge T^+_s\wedge Q_n+  a(k-s).\]
As $Q_n \to (k-s)T^-_{k-s-1}$ in the Hartogs' sense, the first integral converges to $ \int_{\P^k} -u^-_\infty \omega\wedge T^+_s\wedge (k-s)T^-_{k-s-1}$ which is finite hence it is uniformly bounded in $n$ by a constant $C$ (independently of $R$). We apply Stokes and we use $dd^cu^-_\infty = T^--\omega$: 
\begin{align*}\frac{d_s^{n}}{d_{s+1}^{n}} I_2 &\leq \int_{\P^k} -U_R dd^cu^-_\infty \wedge T^+_s\wedge Q_n+  a(k-s)+C \\
	   &\leq \int_{\P^k} -U_R T^- \wedge T^+_s\wedge Q_n+ \int_{\P^k} U_R \omega  \wedge T^+_s\wedge Q_n+  a(k-s)+C \\
	   & \leq \int_{\P^k} -U_R T^- \wedge T^+_s\wedge Q_n+  a(k-s)+C
\end{align*}
since $U_R \leq 0$. Using now $-U_R \leq - u^-_\infty + a$ gives 
\begin{align*}\frac{d_s^{n}}{d_{s+1}^{n}} I_2 & \leq \int_{\P^k} - u^-_\infty T^- \wedge T^+_s\wedge Q_n+  2a(k-s)+C.
\end{align*}
As above, the integral $\int_{\P^k} u^-_\infty T^- \wedge T^+_s\wedge Q_n$ converges to $\int_{\P^k} u^-_\infty T^- \wedge T^+_s\wedge (k-s)T^-_{k-s-1}$ which is bounded hence it is uniformly bounded in $n$ (independently of $R$), hence there exits a constant $C'$, independent of $R$, such that  
\begin{align*}\frac{d_s^{n}}{d_{s+1}^{n}} I_2 & \leq  2a(k-s)+C'.
\end{align*}
Combining the above and $\frac{d_s}{d_{s+1}}=\delta$, we deduce that 
 	\[\left| \int_{\P^k} \frac{1}{d_{s+1}^n} (f^n)^*(dd^cU_R\wedge T^+_s)\wedge U_{T^-_{k-s}}  \right|\leq C_1\delta^n \sqrt{\delta^{-n}(C' +2a(k-s))} \leq K \sqrt{\delta}^n(1+a)\]
 	for $K$ large enough. 
\end{proof}
	
\section{Exponential mixing of all orders and CLT}\label{Section proofs}
	
We prove the following proposition which will imply the mixing of all orders and the CLT.
\begin{proposition}\label{prop-principale}
	For all $p\geq 1$, there exists $C_p >0$ with  $(C_p) \nearrow  $ such that  $ \forall \varphi^0, \dots, \varphi^p \in \mathcal{A}$ and $\forall 0=:n_0\leq n_1 \leq \dots \leq n_p$, we have 
	\begin{align*}&\left|  \langle \mu, \prod_{j=0}^p \varphi^j\circ f^{n_j} \rangle -  \langle \mu, \varphi^0 \rangle. \langle \mu, \prod_{j=1}^p \varphi^j\circ f^{n_j}  \rangle\right| \leq \\
	&	C_p d^{-\frac{s(n_{1}-n_0)}{2k}}\|\varphi^0\|_{\mathcal{A}^+} \|\varphi^1\|_{\mathcal{A}^-} \dots \|\varphi^p\|_{\mathcal{A}^-}. \end{align*}	  
\end{proposition}	
Recall that $\mathcal{A}_\pm$ is stable under composition and its elements are well defined outside a pluripolar set, in particular, the different quantities in the proposition are well defined. We follow the strategy of proof of \cite[Proposition 2]{Vigny_decay}. 
\begin{proof}
	First, by Proposition~\ref{prop_regularization}, the fact that $\mu$ does not charge pluripolar set, and dominated convergence, it is enough to consider the case where all $\varphi_i$ are smooth, which we assume from now on.

	For simplicity, write $m:=  \left[ s  (n_1-n_0)/k\right] $ (where $\left[ .\right]$ denotes the integer part) and $n'_j:=n_j-m$. By invariance of $\mu$
	\[  \langle \mu, \prod_{j=0}^p \varphi^j\circ f^{n_j} \rangle -  \langle \mu, \varphi^0 \rangle. \langle \mu, \prod_{j=1}^p \varphi^j\circ f^{n_j}  \rangle=  \langle \mu, \prod_{j=0}^p \varphi^j\circ f^{n'_j} \rangle -  \langle \mu, \varphi^0 \rangle. \langle \mu, \prod_{j=1}^p \varphi^j\circ f^{n_j} \rangle .\]
We write $\varphi^0\circ f^{-m}= \sum_{i=0}^{-m} c_i + \varphi^0_{-m}$ where the quantities are defined in \eqref{varphi_n} replacing $f$ with $f^{-1}$ so we put $-$ signs. 	By Proposition~\ref{rate_measure} (applied to $f^{-1}$), we have $|\sum_{i=0}^{-m} c_i -\mu(\varphi^0)|\leq C \sqrt{d}^{-m} \|\varphi^0\|_{\mathcal{A}^+} $. Hence:
\begin{align*} 
\langle \mu, \prod_{j=0}^p \varphi^j\circ f^{n_j} \rangle -  \langle \mu, \varphi^0 \rangle. \langle \mu, \prod_{j=1}^p \varphi^j\circ f^{n_j}  \rangle&= \\ \langle \mu, \varphi_{-m}^0\prod_{j=1}^p \varphi^j\circ f^{n'_j} \rangle +  \left(\sum_{i=0}^{-m} c_i-\langle \mu, \varphi^0 \rangle\right) \langle \mu, \prod_{j=1}^p \varphi^j\circ f^{n_j}  \rangle.
\end{align*}
By the above and Remark~\ref{defined_up_to_a_pp} , 
\[\left|\left(\sum_{i=0}^{-m} c_i-\langle \mu, \varphi^0 \rangle\right) \langle \mu, \prod_{j=1}^p \varphi^j\circ f^{n_j}  \rangle\right| \leq C \sqrt{d}^{-m} \|\varphi^0\|_{\mathcal{A}^+} \|\varphi^1\|_{\infty} \dots \|\varphi^p\|_{\infty}.\]
As $m \simeq s(n_1-n_0)/k$, we have $\sqrt{d}^{-m} =O( d^{-s(n_1-n_0)/2k})$ so we are left with  controlling 
\[I:=\langle \mu, \varphi_{-m}^0\prod_{j=1}^p \varphi^j\circ f^{n'_j} \rangle.\] 
 Let $\chi_A$ be the cut-off function defined
earlier for $I(f^{n_p-m})$ (i.e. the largest integer appearing) and let $\xi_A$ be the similar cut-off function associated to $I(f^{-m})$, by dominated convergence
\[I = \lim_{A\to \infty } \langle \mu, \chi_A \xi _A \varphi_{-m}^0\prod_{j=1}^p \varphi^j\circ f^{n'_j} \rangle.\]

 We denote by $(x,y)$ the coordinates on $\P^k \times \P^k$. Recall $\mu= T_s^+ \wedge T_{k-s}^-$ and let $T^+_{s,\varepsilon}$ and $T^-_{k-s, \varepsilon}$ be Hartogs' approximations
 of $T^+_s$ and $T^-_{k-s}$ and $n\geq n'_p= n_p-m$ an integer. By Hartogs convergence $	\frac{1}{d_s^{n}}(f^{n})^*T^+_{s,\varepsilon}\wedge \frac{1}{d_s^{m}}(f^{m})_*T^-_{k-s,\varepsilon}\to \mu$ in the Hartogs'sense (hence in the sense of measures), so we can rewrite $I= \lim_{A\to \infty} \lim_{\varepsilon \to 0} I_{A,\varepsilon} $ where
 \begin{align*} I_{A,\varepsilon}=  \int_{\P^k\times \P^k} \chi_A(x) \xi_A(y) _A\varphi_{-m}^0(y)\prod_{j=1}^p \varphi^j\circ f^{n'_j}(x) 	\frac{1}{d_s^{n}}(f^{n})^*T^+_{s,\varepsilon}(x)\wedge 
 	\frac{1}{d_s^{m}}(f^{m})_*T^-_{k-s,\varepsilon}(y) \wedge [\Delta] 
 \end{align*}
 where $[\Delta]$ is the current of integration on the diagonal $\Delta$ of
 $(\P^k)^2$. We write $[\Delta]= \sum_{i=0}^k \omega^i(x)\wedge \omega^{k-i}(y)+ dd^c V$
 where $V$ is the negative quasi-potential of $\Delta$ given in \cite[Theorem 2.3.1]{superpotentiels}. We infer, by bidegree's arguments that  $I_{A,\varepsilon}$ can be written as
 \begin{align*}
\int_{(\P^k)^2} \chi_A(x) \xi_A(y) \varphi_{-m}^0(y)\prod_{j=1}^p \varphi^j\circ f^{n'_j}(x)  \frac{1}{d_s^{n}}(f^{n})^*T^+_{s,\varepsilon}(x)\wedge  \frac{1}{d_s^{m}}(f^{m})_*T^-_{k-s,\varepsilon}(y) \wedge dd^c V  + \\
\int_{(\P^k)^2} \chi_A(x) \xi_A(y) \varphi_{-m}^0(y)\prod_{j=1}^p \varphi^j\circ f^{n'_j}(x) \frac{1}{d_s^{n}}(f^{n})^*T^+_{s,\varepsilon}(x)\wedge 
\frac{1}{d_s^{m}}(f^{m})_*T^-_{k-s,\varepsilon}(y) \wedge \omega^{k-s}(x)\wedge \omega^{s}(y).
 \end{align*}
 The second integral converges to $\int_{\P^k\times \P^k} \varphi_{-m}^0(y)\prod_{j=1}^p \varphi^j\circ f^{n'_j}(x)  T^+_{s}(x) \wedge T^-_{k-s}(y) \wedge \omega^{k-s}(x)\wedge \omega^{s}(y)$ which is  zero by Fubini ($\int_{\P^k} \varphi_{-m}^0 \omega^s \wedge T^-_{k-s}=0$ by \eqref{varphi_n}). The above can be restated as
 \begin{align}\label{Stokes16}
 	I= \lim_{A\to \infty} \lim_{\varepsilon \to 0} \int_{(\P^k)^2}
 	&\chi_A(x)
 	\xi_A(y) \varphi_{-m}^0(y)\prod_{j=1}^p \varphi^j\circ f^{n'_j}(x) 
 	 	\frac{1}{d_s^{n}}(f^{n})^*T^+_{s,\varepsilon}(x)\\
 	&\wedge \frac{1}{d_s^{m}}(f^{m})_*T^-_{k-s,\varepsilon}(y)\wedge dd^c V. \nonumber
 \end{align}
 Recall results of  \cite{superpotentiels}[Theorem 2.3.1]. Let $p_1$ and $p_2$ be
 the projections from $(\P^k)^2$ onto its factors. For  $R\in \mathcal{C}_p$, we have that $(p_1)_*(p_2^*(R)\wedge V)$ is the Green quasi-potential of $R$
 associated to $V$ (see the definition of $U$ at the last line of \cite{superpotentiels}[p.17]). We
 denote it by $U_{R}$. Denote by
 $\mathcal{U}_R$ the super-potential of $R$ associated to $U_R$. That implies that the mean of
 $\mathcal{U}_R$ is uniformly bounded for all $R$ and depends continuously on $R$. In particular,
 if $(R_\varepsilon)$ and $(S_\varepsilon)$ are sequences of currents in $\mathcal{C}_q$ and
 $\mathcal{C}_{k-q+1}$  that $H$-converge to $R$ and $S$ then $\mathcal{U}_{R_\varepsilon}(S_\varepsilon) \to
 \mathcal{U}_{R}(S)$.

 Pick $A$ and $\varepsilon$. In the integral \eqref{Stokes16}, we apply Stokes formula (observe that by our construction, every term is smooth in the integral apart from the quasi-potential $V$ so everything is well defined). Developing $dd^c\left(\chi_A(x)\xi_A(y) \varphi_{-m}^0(y)\prod_{j=1}^p \varphi^j\circ f^{n'_j}(x)\right)$   leads to an expression containing $(p+3)^2$ terms. Applying verbatim the same arguments than \cite[Lemmas 9 and 10]{Vigny_decay}, every term containing $d\chi_A(x)$, $d^c\chi_A(x)$, $d\xi_A(y)$,  $d^c\xi_A(y)$, $dd^c  \chi_A(x)$, $dd^c \xi_A(y)$ can be taken arbitrarily small for $A$ large enough (independently of $\varepsilon$). So we are left with terms in (recall $d u \wedge d^c v \wedge Q= dv \wedge d^c u \wedge Q$ for a current $Q$ of bidimension $(1,1)$)
 \begin{align*}
 &  \prod_{j=1, j\neq \ell }^p \varphi^j\circ f^{n'_j}(x) 
 		d\varphi_{-m}^0(y) \wedge d^c \varphi^\ell\circ f^{n'_\ell}(x) , \\
& \varphi_{-m}^0(y) \prod_{j=1, j\neq \ell, j \neq \ell' }^p \varphi^j\circ f^{n'_j}(x) (d \varphi^\ell\circ f^{n'_\ell}\wedge d^c \varphi^{\ell'}\circ f^{n'_{\ell'}})(x), \\
 & \prod_{j=1}^p \varphi^j\circ f^{n'_j}(x) 
 dd^c \varphi_{-m}^0(y),\\	
 & \varphi_{-m}^0(y) \prod_{j=1, j\neq \ell }^p \varphi^j\circ f^{n'_j}(x) 
 dd^c \varphi^\ell\circ f^{n'_\ell}(x)  	
\end{align*} 
where $1\leq \ell \neq \ell' \leq p$ are distinct integers. Let us consider an integral containing terms of the first type, by Cauchy-Schwarz inequality \cite[Lemme 1]{Okada}:
 \begin{align*}
&\Big|\int_{(\P^k)^2} \chi_A(x)\xi_A(y)
  \prod_{j=1, j\neq \ell }^p \varphi^j\circ f^{n'_j}(x) 
d\varphi_{-m}^0(y) \wedge d^c \varphi^\ell\circ f^{n'_\ell}(x)  \wedge   \frac{1}{d_s^{n}}(f^{n})^*T^+_{s,\varepsilon}(x) \\
&
\wedge \frac{1}{d_s^{m}}(f^{m})_*T^-_{k-s,\varepsilon}(y)\wedge V\Big|^2 \leq \prod_{j=1, j\neq \ell}^p \|\varphi^j\|^2_\infty \times\\ 
&  \Big|\int_{(\P^k)^2}  \chi_A(x)\xi_A(y)
d\varphi_{-m}^0(y) \wedge d^c \varphi_{-m}^0(y) \wedge  \frac{1}{d_s^{n}}(f^{n})^*T^+_{s,\varepsilon}(x)
\wedge \frac{1}{d_s^{m}}(f^{m})_*T^-_{k-s,\varepsilon}(y)\wedge V\Big|\\
&\Big|\int_{(\P^k)^2}
d  \varphi^\ell\circ f^{n'_\ell}(x) \wedge d^c \varphi^\ell\circ f^{n'_\ell}(x) \wedge  \chi_A(x)\xi_A(y) \frac{1}{d_s^{n}}(f^{n})^*T^+_{s,\varepsilon}(x)
\wedge \frac{1}{d_s^{m}}(f^{m})_*T^-_{k-s,\varepsilon}(y)\wedge V\Big|.
\end{align*} 
 Consider the first integral of the product, by Fubini, it is bounded from above by
 \begin{align*}
\Big|\int_{\P^k}
	|\varphi^0|^+_1 \frac{1}{d_s^{m}}f^{m}_* (R_1^{0,+} \wedge T^-_{k-s,\varepsilon}) 
	\wedge U_{\frac{1}{d_s^{n}}(f^{n})^*T^+_{s,\varepsilon}}\Big|
\end{align*}
where $R_1^{0,+}$  is a smooth current more $(H,a)$-regular than $T^+$ with  $d\varphi^0\wedge d^c \varphi^0 \leq 	|\varphi^0|^+_1 (R_1^{0,+})$. In term of super-potentials, we can write it as
  \begin{align*}
 	-|\varphi^0|^+_1 \frac{d^m_{s-1}}{d_s^{m}} \mathcal{U}_{\frac{1}{d_s^{n}}(f^{n})^*T^+_{s,\varepsilon}} \left( \frac{1}{d_{s-1}^{m}}f^{m}_* (R_1^{0,+} \wedge T^-_{k-s,\varepsilon})\right). 
 \end{align*}
Since everything converges in the Hartogs sense when $\varepsilon \to 0$, the continuity in the Hartogs sense of the wedge product and of the push-forward for admissible currents, we have that this quantity converges to 
    \begin{align*}
   	-|\varphi^0|^+_1 \frac{d^m_{s-1}}{d_s^{m}} \mathcal{U}_{T^+_{s}} \left( \frac{1}{d_{s-1}^{m}}f^{m}_* (R_1^{0,+} \wedge T^-_{k-s})\right). 
   \end{align*}
By \eqref{eqbound_U+} (notice that the super-potential we use is not the same than the one in \eqref{eqbound_U+}, but as two super-potentials differ by a constant, this is not an issue)  and the fact that  $d = \frac{d_{s}}{d_{s-1}}$, this term is indeed bounded by  $C |\varphi^0|^+_1 d^{-m/2 }$ for some constant $C$ independent on $m$. Going back to the second integral of the above product, by Fubini, it is bounded from above by
 \begin{align*}
	\Big|\int_{\P^k}
	|\varphi^\ell|^-_1 \frac{1}{d_s^{n'_\ell}}(f^{n'_\ell})^* \left(R_1^{\ell,-} \wedge \frac{1}{d_s^{n-n'_\ell}}(f^{n-n'_\ell})^*T^+_{s,\varepsilon}\right) 
	\wedge U_{\frac{1}{d_s^{m}}(f^m)_*(T^-_{k-s,\varepsilon})}\Big|
\end{align*}
where $R_1^{\ell,-}$  is more $(H,a)$-regular than $T^-$ with  $d\varphi^\ell\wedge d^c \varphi^\ell \leq 	|\varphi^\ell|^-_1 (R_1^{\ell,-})$. Again, convergence in the Hartogs sense allows to interpret its limit as $\varepsilon\to 0$ as 
\[-\frac{d_{s+1}^{n'_\ell}}{d_s^{n'_\ell}}
|\varphi^\ell|^-_1 \mathcal{U}_{T^-_{k-s}}\left(\frac{1}{d_{s+1}^{n'_\ell}}(f^{n'_\ell})^* (R_1^{\ell,-} \wedge T^+_{s}) \right).  \]
This time, \eqref{eqbound_U-}  and the fact that  $\delta = \frac{d_{s}}{d_{s+1}}$, this term is bounded by  $C |\varphi^\ell|^-_1 \delta^{-n'_\ell/2}\leq  C |\varphi^\ell|^-_1  d^{-\frac{s(n_{1}-n_0)}{2k}}$ as $n'_\ell=n_\ell -m \geq  n_{1}-m$ and $d^s =\delta^{k-s}$.

 Consider now the square of the term in $(d \varphi^\ell\circ f^{n'_\ell}\wedge d^c \varphi^{\ell'}\circ f^{n'_{\ell'}})(x )$. Using Cauchy-Schwarz inequality  as above will give again two terms, one is the last one we controlled and the other is the same term for $n'_{\ell'}$ instead of $n'_\ell$ hence it is also  
 \[ \leq C |\varphi^\ell|^-_1|\varphi^{\ell'}|^-_1 \prod_{j \neq \ell, j\neq \ell'}  \|\varphi^0_{-m}\|^2_\infty \|\varphi_j\|^2_\infty d^{-\frac{s(n_{1}-n_0)}{k}} .\]
 which bounds this term since $ \|\varphi^0_{-m}\|_\infty \leq 2  \|\varphi^0\|_\infty$ (and we take the square root at the end).
 
 Consider the integral containing the term in  $\prod_{j=1}^p \varphi^j\circ f^{n'_j}(x) 
 dd^c \varphi_{-m}^0(y)$:
 \begin{align*}
 	&\Big|\int_{(\P^k)^2}\prod_{j=1}^p \varphi^j\circ f^{n'_j}(x) 
 	dd^c \varphi_{-m}^0(y) \wedge  \chi_A(x)\xi_A(y) \frac{1}{d_s^{n}}(f^{n})^*T^+_{s,\varepsilon}(x)
 	\wedge \frac{1}{d_s^{m}}(f^{m})_*T^-_{k-s,\varepsilon}(y)\wedge V\Big|.
 \end{align*}
 Let $R_2^{0,+} \in \mathcal{C}_1$, more $(H,a)$-regular than $T^+$ such that  $0\leq |\varphi^0|_{2}^+R_2^{0,+} \pm dd^c \varphi^0$. In particular,  
 \begin{align*}
 \prod_{j=1}^p\varphi^j\circ f^{n'_j}(x) 
 	dd^c \varphi_{-m}^0(y)   \chi_A(x)\xi_A(y) = \left\|\prod_{j=1}^p\varphi^j\circ f^{n'_j}(x) \chi_A(x) \right\|_\infty  \xi_A(y) 
 	dd^c \varphi_{-m}^0(y)  \\
 	+ \left(\prod_{j=1}^p\varphi^j\circ f^{n'_j}(x) \chi_A(x)   
 	- \left\|\prod_{j=1}^p\varphi^j\circ f^{n'_j}(x) \chi_A(x) \right\|_\infty \right)\xi_A(y)    
 	dd^c \varphi_{-m}^0(y)\\
 \leq \prod_{j=1}^p \|\varphi^j\|_\infty   \xi_A(y)
 (f^m)_* (|\varphi^0|_{2}^+R_2^{0,+})  \\
 - \left(\prod_{j=1}^p\varphi^j\circ f^{n'_j}(x) \chi_A(x) 
 -\prod_{j=1}^p \|\varphi^j\|_\infty    \right) \xi_A(y) (f^m)_* (|\varphi^0|_{2}^+R_2^{0,+}) \\
  \leq  3 |\varphi^0|_{2}^+ \prod_{j=1}^p \|\varphi^j\|_\infty   
 \xi_A(y)(f^m)_* (R_2^{0,+})
 \end{align*}
 and similarly   
  \begin{align*}
 	- \prod_{j=1}^p\varphi^j\circ f^{n'_j}(x) 
 	dd^c \varphi_{-m}^0(y)   \chi_A(x)\xi_A(y) 
 	\leq  3 |\varphi^0|_{2}^+ \xi_A(y) \prod_{j=1}^p \|\varphi^j\|_\infty   
 	(f^m)_* (R_2^{0,+}).
 \end{align*}
 So using that $V \leq 0$ in the sense of currents gives
 \begin{align*}
	&\Big|\int_{(\P^k)^2}\prod_{j=1}^p \varphi^j\circ f^{n'_j}(x) 
	dd^c \varphi_{-m}^0(y) \wedge  \chi_A(x)\xi_A(y) \frac{1}{d_s^{n}}(f^{n})^*T^+_{s,\varepsilon}(x)
	\wedge \frac{1}{d_s^{m}}(f^{m})_*T^-_{k-s,\varepsilon}(y)\wedge V\Big| \leq \\
	&-\int_{\P^k}   \frac{3}{d_s^{m}} |\varphi^0|_{2}^+ \prod_{j=1}^p \|\varphi^j\|_\infty   
	(f^m)_* (R_2^{0,+} \wedge T^-_{k-s, \varepsilon}) \wedge U_{ \frac{1}{d_s^{n}}(f^{n})^*T^+_{s,\varepsilon}} 
\end{align*} 
so processing as above gives that, letting $\varepsilon\to 0$,  it is again $\leq C  |\varphi^0|^+_2 \prod_{j=1}^p \|\varphi^j\|_\infty   d^{-m/2 }$. Terms in $\varphi_{-m}^0(y) \prod_{j=1, j\neq \ell }^p \varphi^j\circ f^{n'_j}(x) 
dd^c \varphi^\ell\circ f^{n'_\ell}(x) $ are similar and gives bound in $\leq C  |\varphi^\ell|^-_2 \ \prod_{j=0, j\neq \ell}^p \|\varphi^j\|_\infty   d^{-m/2 }$ which concludes the proof (replace $C_p$ with $\max( C_{p}, C_{p-1})$ to make it increasing if necessary).
\end{proof}	
	
\begin{corollary}[exponential mixing of all orders] For all $ p\geq 1$, there exists a constant $C_p >0$ such that $ \forall \varphi^0, \dots, \varphi^p \in \mathcal{A}$ and $0=:n_0\leq n_1 \leq \dots \leq n_p$, we have 
	\begin{align*}
		&\left|  \langle \mu, \prod_{j=0}^p \varphi^j\circ f^{n_j} \rangle -   \prod_{j=0}^p\langle \mu, \varphi^j\rangle \right|\leq 	C_p d^{-\frac{s}{2k}\min_{0\leq j \leq p-1}(n_{j+1}-n_j)}\|\varphi^0\|_{\mathcal{A}} \|\varphi^1\|_{\mathcal{A}} \dots \|\varphi^p\|_{\mathcal{A}}. 
		\end{align*}		
\end{corollary}	
	\begin{proof} Write 
	\begin{align*}
I&:=\left|  \langle \mu, \prod_{j=0}^p \varphi^j\circ f^{n_j} \rangle -   \prod_{j=0}^p\langle \mu, \varphi^j\rangle \right| \\
  &=  \left| \sum_{i=0}^{p}  \left(\prod_{j=0}^{i-1}\langle \mu, \varphi^j\rangle  \langle \mu, \prod_{j=i}^p \varphi^j\circ f^{n_j} \rangle - \prod_{j=0}^{i}\langle \mu, \varphi^j\rangle  \langle \mu, \prod_{j=i+1}^p \varphi^j\circ f^{n_j} \rangle\right) \right|,
   \end{align*}
   with the natural convention that an empty product is $1$. By  Proposition~\ref{prop-principale} and Remark~\ref{defined_up_to_a_pp}
  	\begin{align*}
  	I&\leq \sum_{i=0}^{p}  \left(\prod_{j=0}^{i-1}|\langle \mu, \varphi^j\rangle|  \left| \langle \mu, \prod_{j=i}^p \varphi^j\circ f^{n_j} \rangle -\langle \mu, \varphi^i\rangle  \langle \mu, \prod_{j=i+1}^p \varphi^j\circ f^{n_j} \rangle \right|\right) \\
  	&\leq \sum_{i=0}^{p}  \left(\prod_{j=0}^{i-1}\|\varphi^j\|_\infty  \left| \langle \mu, \prod_{j=i}^p \varphi^j\circ f^{n_j-n_i} \rangle -\langle \mu, \varphi^i\rangle  \langle \mu, \prod_{j=i+1}^p \varphi^j\circ f^{n_j-n_i} \rangle \right|\right)\\
  	&\leq \sum_{i=0}^{p}  \left(\prod_{j=0}^{i-1}\|\varphi^j\|_\infty\right) C_{p-i}d^{-\frac{s(n_{i+1}-n_i)}{2k}} \|\varphi^i\|_{\mathcal{A}^+} \|\varphi^{i+1}\|_{\mathcal{A}^-} \dots \|\varphi^p\|_{\mathcal{A}^-}\\
&  	\leq (p+1)C_p  d^{-\frac{s}{2k}\min_{0\leq j \leq p-1}(n_{j+1}-n_j)} \|\varphi^0\|_{\mathcal{A}} \|\varphi^{1}\|_{\mathcal{A}} \dots \|\varphi^p\|_{\mathcal{A}},
  \end{align*} 
  since $C_p$ is increasing. The result follows (renaming $C_p$).
	\end{proof}
\begin{proof}[Proof of Theorem~\ref{tm_exp}]
	The proof is now an immediate application of the previous corollary with the bounds of the norm in $\mathcal{A}^a$ of Examples~\ref{ex_smooth}  for $\varphi_j \in C^2$ and \ref{ex_DSH} for $\varphi_j \in \mathrm{DSH}_a^\infty(\P^k)$. For the Hölder case, the result follows from the $C^2$ case by iterating an interpolation argument (see for instance \cite[pp. 262-263]{Dinh_decay}).
\end{proof}

Write $\kappa:=\frac{s\log d}{2k}$ so that $d^{-\frac{s}{2k}\min_{0\leq j \leq p-1}(n_{j+1}-n_j)}= e^{-\kappa\min_{0\leq j \leq p-1}(n_{j+1}-n_j) }$. \\

The proof of the CLT is a consequence of \cite[Lemma 9.1]{BG} where it is the only place the invariance of the space of test functions under the dynamics is used. As mentioned above, in our case, the space $\mathcal{A}$ is not invariant under $f$ (only $\mathcal{A}^+$ is but then it is not invariant for $f^{-1}$). Careful analysis of the proof shows that in fact, Proposition~\ref{prop-principale} still gives the estimate. We give the details for the sake of the reader. Recall first the lemma. 
\begin{corollary}[Lemma 9.1 of \cite{BG} holds]\label{lemma 9.1} Fix $0\leq \alpha < \beta$. Then, for every partition $\Omega$ of $[r]:=\{1,\dots,r\}$, $\underline{h}\in \Delta_\Omega(\alpha,\beta)$ and $\varphi \in \mathcal{A}$:
	\[\left|\Psi_{\varphi, \underline{h}}(I)- \Psi^\Omega_{\varphi, \underline{h}}(I)\right|\ll_r e^{-(\beta \kappa  - r \alpha \max (\log d, \log \delta))}\|\varphi\|_{\mathcal{A}}^{|I|}\]
		for every $I\subset [r]$. Here, $\ll_r$ means that the constant depends only on $r$ (and not on $\alpha$ nor $\beta$).
\end{corollary}	
Let us explain the notations first and their meaning in our setting where the group acting is $\{f^n\}_{n\in \Z}$. An element $\underline{h}\in \Delta_\Omega(\alpha,\beta)$ is a $r$-tuple $(f^{n_1}, \dots, f^{n_r})$ with $(n_1, \dots n_r)\in \Z^r$ with the following conditions	
\begin{itemize}
	\item $d^{\Omega}(\underline{h})\leq \alpha$ where $d^{\Omega}(\underline{h})= \max \{d^J(\underline{h}), \ J\in \Omega \}= \max \{ (\max |n_i-n_j|, \ i,j\in J), \ J\in \Omega \}$. In other words,  $d^{\Omega}(\underline{h})\leq \alpha$ means that for any element $J$ of the partition, the iterations appearing in $J$ are close from each other. 
	\item $d_\Omega(\underline{h})>\beta$ where $d_\Omega(\underline{h})= \min \{ d_{J,J'}(\underline{h}), \ J,J'\in \Omega, \ J\neq J'\}=\min \{ (\min |n_j-n_{j'}|, \ j\in J,\ j'\in J'), \ J,J'\in \Omega, \ J\neq J'\}$.  In other words,  $d_\Omega(\underline{h})>\beta$ means that the iteration numbers belonging to distinct elements of the partition are far from each other.  	
\end{itemize}	
As in \cite{BG}, denote $W_I:= \{J\in \Omega, I\cap J \neq \varnothing\}= \{ J_1, J_2, \dots, J_\ell\}$ where, by definition, $J_1$ contains the element $i_1\in I$ such that $n_{i_1}$ is the largest element in $\underline{h}$, $J_2$ contains the element $i_2\in I$ such that $n_{i_2}$ is the largest element in $\underline{h}$ not corresponding to $n_i$ with $i\in J_1$ and so on. 

Let $\varphi_m:= \prod_{i \in I \cap J_m} \varphi \circ f^{-n_i}$ (observe that in \cite{BG}, $h.\varphi= \varphi \circ h^{-1}$). Then (see \cite[definition p.470]{BG})
\[ \Psi_{\varphi, \underline{h}}(I \cap J_m) = \langle \mu, \varphi_m\rangle,    \]
\[ \Psi_{\varphi, \underline{h}}(I) = \langle \mu, \prod_{m=1}^\ell  \varphi_m\rangle.   \]
With the convention that $\Psi_{\varphi, \underline{h}}(\varnothing)=1$, we have (see \cite[§7.2]{BG})
\[ \Psi^\Omega_{\varphi, \underline{h}}(I) :=  \prod_{m=1}^\ell \Psi_{\varphi, \underline{h}}(I \cap J_m) = \prod_{m=1}^\ell \langle \mu,  \varphi_m\rangle.    \]	
\begin{proof}
{\bf Step 1.} For $m=1$, take $n_1^1:=n_{i_1}$ and, for $m=2, \dots, \ell$, $n_m^1$ is the smallest possible integer in $\underline{h}$ for indexes in $J_m\cap I$. Denote
\[\varphi_m^1:= \prod_{i\in I\cap J_m} \varphi \circ f^{-n_i}\circ f^{n^1_m}= \varphi_m\circ f^{n^1_m}, \quad \ m=1,\dots, \ell.\]
By the invariance of $\mu$
\begin{align*}
I_1:&= \left|\Psi_{\varphi, \underline{h}}(I)- \Psi^\Omega_{\varphi, \underline{h}}(I)\right|=\left|    \langle \mu,  \prod_{m=1}^\ell \varphi_m\rangle - \prod_{m=1}^\ell\langle \mu,   \varphi_m\rangle \right|\\
& = \left|    \langle \mu,  \prod_{m=1}^\ell \varphi^1_m\circ f^{-n_m^1}\rangle - \prod_{m=1}^\ell\langle \mu,   \varphi_m\rangle \right| \\
&\leq \left|    \langle \mu,  \prod_{m=1}^\ell \varphi^1_m\circ f^{n_1^1-n_m^1}\rangle - \langle \mu,   \varphi^1_1\rangle\langle \mu,  \prod_{m=2}^\ell \varphi^1_m\circ f^{n_1^1-n_m^1}\rangle\right| \\
&\quad + \left| \langle \mu,   \varphi^1_1\rangle\langle \mu,  \prod_{m=2}^\ell \varphi^1_m\circ f^{-n_m^1}\rangle- \prod_{m=1}^\ell\langle \mu,   \varphi_m\rangle \right|.  
\end{align*}
For the first term of the sum, using Proposition~\ref{prop-principale} for the $(n_1^1-n_m^1)_{ m=1, \dots,\ell }$ (recall that $n_1^1$ is the largest of all the $n_i^m$), it is bounded from above by 
\[C_\ell e^{-\kappa \min_{m\neq 1} |n_m^1-n_{1}^1|}\|\varphi_1^1\|_{\mathcal{A}^+} \|\varphi_2^1\|_{\mathcal{A}^-} \dots \|\varphi_\ell^1\|_{\mathcal{A}^-},\] 
as $C_\ell \geq C_{\ell -1 }$. By the multiplicative property of the norms and Proposition~\ref{prop_stability}:
\begin{align*}
	\| \varphi_1^1\|_{\mathcal{A}^+} &= 	\left\|  \prod_{i \in I\cap J_1} \varphi\circ f^{-n_i+n_1^1}\right\|_{\mathcal{A}^+} 
\leq  20^{|I\cap J_1|} \prod_{i \in I\cap J_1} 	\| \varphi\circ f^{n_1^1-n_i}\|_{\mathcal{A}^+}\\
&\leq 	 20^{|I\cap J_1|} \prod_{i \in I\cap J_1}   d^{n_1^1-n_i}	\| \varphi\|_{\mathcal{A}^+}\leq (20\| \varphi\|_{\mathcal{A}^+})^{|I\cap J_1|} d^{\alpha |I\cap J_1|} 	\end{align*}
by the definition of $\alpha$. For $m\geq 2$, the same computations yield ($-n_i+n_m^1 \leq 0 $ for $i \in J_m\cap I$):
\begin{align*}
	\| \varphi_m^1\|_{\mathcal{A}^-} &= 	\left\|  \prod_{i \in I\cap J_m} \varphi\circ f^{-n_i+n_m^1}\right\|_{\mathcal{A}^-} 
	\leq (20\| \varphi\|_{\mathcal{A}^-})^{|I\cap J_m|} \delta^{\alpha |I\cap J_m|} .\end{align*}
Thus, since $\langle \mu,   \varphi^1_1\rangle= \langle \mu,   \varphi_1\rangle$ and $|I|=|I\cap J_1|+\dots + |I\cap J_\ell|$:
\begin{align*}
	I_1&\leq C_\ell 20^{|I|}  e^{-\kappa \beta} \|\varphi\|_{\mathcal{A}^+}^{|J_1 \cap I|} \|\varphi\|_{\mathcal{A}^-}^{|J_2 \cap I|} \dots \|\varphi\|_{\mathcal{A}^-}^{|J_\ell \cap I|} \max(d,\delta)^{\alpha |I|} \\   
&\quad + \left| \langle \mu,   \varphi^1_1\rangle\langle \mu,  \prod_{m=2}^\ell \varphi^1_m\circ f^{-n_m^1}\rangle- \prod_{m=1}^\ell\langle \mu,   \varphi_m\rangle \right|\\
&\leq C_\ell 20^{|I|}  e^{-\kappa \beta} \|\varphi\|_{\mathcal{A}}^{|I|} \max(d,\delta)^{\alpha |I|}+  |\langle \mu,   \varphi_1\rangle|\left|\langle \mu,  \prod_{m=2}^\ell \varphi_m\rangle- \prod_{m=2}^\ell\langle \mu,   \varphi_m\rangle \right|.
\end{align*}

\noindent {\bf Step 2.} We repeat the same argument for
 \[I_2:=\left|\langle \mu,  \prod_{m=2}^\ell \varphi_m\rangle- \prod_{m=2}^\ell\langle \mu,   \varphi_m\rangle \right|.\]
  For $m=2$, we pick $n_2^2$ to be the largest integer in $\underline{h}$ corresponding to indexes in $J_2\cap I$ (so $n_2^2=n_{i_2}$ with the above notations) and for $m=3,\dots, \ell$, we pick $n_m^2$ to be the smallest integer in $\underline{h}$ corresponding to indexes in  $J_m\cap I$. For $2\leq m\leq \ell$, we define $\varphi_m^2:= \varphi_m\circ f^{n_m^2}= \prod_{i\in J_m \cap I} \varphi\circ f^{-n_i+n_m^2}$. That way
  \begin{align*}
  	I_2&= \left|\langle \mu,  \prod_{m=2}^\ell \varphi^2_m\circ f^{-n_m^2}\rangle- \prod_{m=2}^\ell\langle \mu,   \varphi^2_m\circ f^{-n_m^2}\rangle \right| \\
  	&\leq \left|\langle \mu,  \prod_{m=2}^\ell \varphi^2_m\circ f^{-n_m^2}\rangle - \langle \mu,   \varphi^2_2\rangle  \prod_{m=3}^\ell\langle \mu,   \varphi^2_m\circ f^{-n_m^2}\rangle \right|\\
&\quad +\left|\langle \mu,   \varphi^2_2\rangle  \prod_{m=3}^\ell\langle \mu,   \varphi^2_m\circ f^{-n_m^2}\rangle  	
  	- \prod_{m=2}^\ell\langle \mu,   \varphi^2_m\circ f^{-n_m^2}\rangle \right|. 
  \end{align*} 
We can apply Step 1 to the first term of the sum, replacing $I$ with $I'=I\backslash J_1$ and taking $J'_m=J_{m+1}$ for $m=2, \dots, \ell$. This leads to  (again, $C_{\ell-2}\leq C_\ell$)
\begin{align*}
\left|\langle \mu,  \prod_{m=2}^\ell \varphi^2_m\circ f^{-n_m^2}\rangle - \langle \mu,   \varphi^2_2\rangle  \prod_{m=3}^\ell\langle \mu,   \varphi^2_m\circ f^{-n_m^2}\rangle \right|\leq 
	C_{\ell} 20^{|I'|}  e^{-\kappa \beta} \|\varphi\|_{\mathcal{A}}^{|I'|} \max(d,\delta)^{\alpha |I'|}.
\end{align*} 
By Remark~\ref{defined_up_to_a_pp},  $|\langle \mu,   \varphi_1\rangle|\leq \|\varphi\|_\infty^{|I_1|} \leq \|\varphi\|_\mathcal{A}^{|I_1|}$, this leads to
\begin{align*}
	I_1&\leq C_\ell 20^{|I|}  e^{-\kappa \beta} \|\varphi\|_{\mathcal{A}}^{|I|} \max(d,\delta)^{\alpha |I|} +\|\varphi\|_\mathcal{A}^{|I_1|} 	C_{\ell} 20^{|I'|}  e^{-\kappa \beta} \|\varphi\|_{\mathcal{A}}^{|I'|} \max(d,\delta)^{\alpha |I'|}\\
	& + \|\varphi\|_\mathcal{A}^{|I_1|}  \left|\langle \mu,   \varphi^2_2\rangle  \prod_{m=3}^\ell\langle \mu,   \varphi^2_m\circ f^{-n_m^2}\rangle  	
	- \prod_{m=2}^\ell\langle \mu,   \varphi^2_m\circ f^{-n_m^2}\rangle \right| .
\end{align*} 
So, a straightforward induction gives 
\[I_1\leq \ell C_\ell  20^{|I|} e^{-\kappa \beta} \|\varphi\|_\mathcal{A}^{|I|}\max(d,\delta)^{\alpha |I|}\]
as $|I|\leq r$,   $ \ell C_\ell  20^{|I|}\leq r C_r 20^r$ is indeed a constant that depends only on $r$ and we have
 \[I_1\ll_r e^{-(\beta \kappa  - r \alpha \max (\log d, \log \delta))}\|\varphi\|_{\mathcal{A}}^{|I|}. \]
\end{proof}

\bibliographystyle{alpha}
\bibliography{biblio}

\end{document}